\documentclass[11pt]{amsart}

\usepackage[T1]{fontenc}
\usepackage[cp1250]{inputenc}

\usepackage[arrow, matrix, curve]{xy}
\xyoption{frame}
\usepackage{amsmath,amsthm,amssymb}
\usepackage{amscd,graphics}

\usepackage{graphicx}
\usepackage{fancybox}

\DeclareMathOperator{\clsp}{\overline{span}}

\DeclareMathOperator{\Aut}{Aut}

\renewcommand{\H}{\widetilde H}
\renewcommand{\S}{\widetilde S}
\newcommand{\I}{\widetilde I}
\newcommand{\B}{\mathcal B}

\newcommand{\II}{\mathcal I}

\newcommand{\K}{\mathcal K}

\newcommand{\x}{\widetilde x}

\newcommand{\OO}{\mathcal{O}}

\newcommand{\X}{\widetilde{X}}

\newcommand{\U}{\widetilde U}

\newcommand{\F}{\widetilde F}

\newcommand{\TDelta}{\widetilde \Delta}
\newcommand{\tphi}{\widetilde \varphi}

\newcommand{\al}{\alpha}

\renewcommand{\L}{\mathcal L}

\newcommand{\C}{\mathbb C}

\newcommand{\Z}{\mathbb Z}
\newcommand{\N}{\mathbb N}
\newcommand{\T}{\widetilde{T}}
\newcommand{\dyna}{\texttt{$C^*$-dyna}}

\newtheorem{thm}{Theorem}[section]
\newtheorem{lem}[thm]{Lemma}
\newtheorem{prop}[thm]{Proposition}
\newtheorem{cor}[thm]{Corollary}
\theoremstyle{definition}

\newtheorem{defn}[thm]{Definition}
\newtheorem{ex}[thm]{Example}
\newtheorem{rem}[thm]{Remark}

\begin{document}
 \thispagestyle{empty}
   \title[Ideal structure via reversible extensions]{Ideal structure of crossed products by endomorphisms via reversible extensions of $C^*$-dynamical systems}

\author{B. K.  Kwa\'sniewski}

\address{Department of Mathematics and Computer Science, The University of Southern Denmark, 
Campusvej 55, DK--5230 Odense M, Denmark}
%\email{kwasniewski@imada.sdu.dk}
\address{Institute of Mathematics, Polish Academy of Science,  ul. \'Sniadeckich 8, PL-00-956 Warszawa, Poland}
\address{Institute of Mathematics,  University  of Bialystok\\
ul. Akademicka 2,  PL-15-267  Bialystok,   Poland}
\email{bartoszk@math.uwb.edu.pl}
%\urladdr{http://math.uwb.edu.pl/~zaf/kwasniewski}

\keywords{$C^*$-algebra, endomorphism,  crossed product, ideal structure, simplicity, reversible extension}
\subjclass[2010]{46L05}

\thanks{This research was supported by  NCN  grant number  DEC-2011/01/D/ST1/04112 and by a Marie Curie-Sklodowska Intra European Fellowship within the 7th European Community Framework Programme, project: OperaDynaDual.}

\begin{abstract} We consider an extendible endomorphism $\alpha$ of a $C^*$-algebra $A$. We associate to it a canonical $C^*$-dynamical  system $(B,\beta)$ that extends $(A,\alpha)$ and is `reversible' in the sense  that the endomorphism $\beta$ admits a unique regular transfer operator $\beta_*$. The theory for $(B,\beta)$ is  analogous to the theory of classic crossed products by automorphisms, and the key idea  is to describe the counterparts of classic notions for $(B,\beta)$ in terms of the initial system $(A,\alpha)$. 

We apply this idea to study the ideal structure of  a non-unital version of the crossed product $C^*(A,\alpha,J)$     introduced recently by the author and A. V. Lebedev. This crossed product depends on the choice of an ideal $J$ in $(\ker\alpha)^\bot$, and if $J=(\ker\alpha)^\bot$ it is  a modification of Stacey's crossed product that works well with non-injective $\alpha$'s. 

We provide descriptions of the  lattices of ideals in $C^*(A,\alpha,J)$ consisting of  gauge-invariant ideals and ideals generated by their intersection with $A$. We  investigate conditions under which these lattices coincide with the set of all ideals in $C^*(A,\alpha,J)$. In particular, we obtain simplicity criteria that besides minimality of the action require either outerness of powers of $\alpha$ or pointwise quasinilpotence of $\alpha$.

\end{abstract}
\maketitle

 \tableofcontents

\section{Introduction.}

The theory of  crossed products  of  $C^*$-algebras by endomorphisms was initiated by the ideas of Cuntz \cite{cuntz}. The original constructions were spatial, involved injective  endomorphisms  
and one of their main aims was to produce  new examples of simple $C^*$-algebras \cite{Paschke}, \cite{cuntz2}, \cite{Rordam}, \cite{Murphy}, \cite{Murphy2}, \cite{Szwajcar}. A universal definition of a crossed product by  an   endomorphism was given by Stacey \cite{Stacey}. When   adapted to  semigroup context   \cite{bkr}, \cite{adji} it was investigated by many authors in   connection with Toeplitz algebras of semigroups of isometries \cite{alnr}, \cite{quasilat}, Bost-Connes Hecke algebras arising from number fields \cite{bc-alg}, \cite{hecke5},  phase transitions   \cite{diri} or short exact sequences and tensor products \cite{larsen}. 

Despite of these achievements one has to note that  the universal  definition proposed by Stacey, even though it makes sense for an arbitrary endomorphism, is not correct when the underlying endomorphism is not injective. 
The reason is that Stacey's (multiplicity one) crossed product $A\times_\alpha^1 \N$ is isometric. This means, cf. \cite[Definition 3.1]{Stacey}, that  the endomorphism $\alpha$  implementing the dynamics on the $C^*$-algebra $A$ is represented in $A\times_\alpha^1 \N$  via the relation 
\begin{equation}\label{fundamental covariance relation}
\iota(\alpha(a))= u  \iota(a) u ^*, \qquad a\in A,
\end{equation}
where $u$ is an \emph{isometry} in $M(A\times_\alpha^1 \N)$, and  $\iota:A\to A\times_\alpha^1 \N$ is the universal homomorphism. 
This readily implies that $\iota$ cannot be injective when $\alpha$ is not injective. Accordingly, the crossed product $A\times_\alpha^1 \N$ in fact  does not depend on $(A,\alpha)$ but on the `smaller' quotient system $(A/R, \alpha_R)$ where $R=\ker\iota$ and $\alpha_R(a+R)=\alpha(a)+ R$ is a monomorphism of $A/R$  (cf. Remark \ref{remark to mentioned in the introduction} below). In other words, speaking about non-injective endomorphisms in this context has only formal character.

In the case $A$ is unital, the authors of \cite{kwa-leb} proposed   an alternative construction of a $C^*$-algebra generated  by a \emph{faithful} copy of $A$ and an operator $u$ satisfying \eqref{fundamental covariance relation}, so that  $u$ is an isometry exactly when $\alpha$ is a monomorphism. The point of departure in \cite{kwa-leb} is the remark  that the relation \eqref{fundamental covariance relation} imply that 
\begin{equation}\label{fundamental consequence of covariance relation}
u \textrm{ is a (power) partial isometry} \qquad \textrm{ and } \qquad u^*u\iota(a)=\iota(a)u^*u,\,\,\,\, a\in A.
\end{equation}
The universal $C^*$-algebra generated by objects satisfying \eqref{fundamental covariance relation}, and hence automatically \eqref{fundamental consequence of covariance relation}, was introduced and studied (in the semigroup case) in  \cite{Lin-Rae}. However,  this algebra  should be viewed as a certain Toeplitz extension of the $C^*$-algebra we seek; in particular it is not a generalization of the classic crossed product. As argued in \cite{kwa-leb}, in order to obtain a smaller algebra one should impose an additional relation. This relation can be phrased in terms of objects appearing in \eqref{fundamental consequence of covariance relation} via the formula 
\begin{equation}\label{covariance ideal first appearance}
J=\{a\in A: u^*u \iota(a)=\iota(a)\}
\end{equation}
 where $J$ is an ideal  in $A$. Stipulating that $\iota$ is faithful, one deduces that $J$ must be contained in the annihilator $(\ker\alpha)^\bot$ of the kernel of $\alpha$. For any ideal $J$ in $(\ker\alpha)^\bot$ the crossed product $C^*(A,\alpha,J)$ studied in \cite{kwa-leb} is a universal $C^*$-algebra with respect to relations \eqref{fundamental covariance relation}, \eqref{covariance ideal first appearance}.
The appropriate modification of Stacey's crossed product  is the $C^*$-algebra $C^*(A,\alpha):=C^*(A,\alpha,(\ker\alpha)^\bot)$ corresponding to $J=(\ker\alpha)^\bot$.  Significantly, see \cite{kwa-leb},    the $C^*$-algebras $C^*(A,\alpha,J)$ can be naturally modeled as relative Cuntz-Pimsner algebras \cite{ms}. In this picture Stacey's crossed product corresponds to Pimsner's original construction \cite{p} while $C^*(A,\alpha)$ corresponds  to Katsura's  \cite{katsura1} modification of Pimsner's  $C^*$-algebras.

We also note that there are good reasons for studying $C^*(A,\alpha,J)$ with $J$ varying from $\{0\}$ to $(\ker\alpha)^\bot$ rather than focusing only on the unrelative crossed product $C^*(A,\alpha)$. For instance: 1) when $J=\{0\}$ we arrive at the partial isometric crossed product of \cite{Lin-Rae}; 2) as we will see below, quotients of the unrelative crossed product by gauge-invariant ideals in general  have the form $C^*(A,\alpha,J)$ where $J\subsetneq (\ker\alpha)^\bot$; 3) compressions of  weighted composition operators often generate $C^*$-algebras  of type $C^*(A,\alpha,J)$  where the dependence  on $J$ is essential, cf.  \cite[Theorem. 5.6]{kwa-logist}.
\medskip

In the present paper, motivated by  applications to non-unital $C^*$-algebras, cf. for instance  \cite{hr}, \cite{larsen}, we extend the definition   of  $C^*(A,\alpha,J)$. In our setting  it applies to   \emph{$C^*$-dynamical systems} $(A,\alpha)$ consisting of a (possibly non-unital) $C^*$-algebra $A$ and an endomorphism $\alpha:A\to A$ that extends to the strictly continuous endomorphism $\overline{\alpha}:M(A)\to M(A)$, cf. \cite{adji}. We call $(A,\alpha)$ a \emph{reversible $C^*$-dynamical system}   if $\alpha$ shifts $(\ker\alpha)^\bot$ onto a corner in $A$ and $\ker\alpha$ is a complemented ideal in $A$. Then there exists a unique  transfer operator $\alpha_*$ for $\alpha$ \cite{exel2}, \cite{brv}, such that $\alpha_*\circ \alpha$ is a conditional expectation on  $\alpha(A)$, cf. \cite{kwa-trans}, \cite{kwa-exel}. For such systems the $C^*$-algebra $C^*(A,\alpha)$ coincides with Exel's crossed product $A\rtimes_{\alpha,\alpha_*}\N$, introduced in \cite{exel2} and  adapted to non-unital case  in \cite{brv}, see also \cite{kwa-exel}. The structure of $C^*(A,\alpha)=A\rtimes_{\alpha,\alpha_*}\N$ is relatively similar to that of the classic crossed product by an automorphism, see \cite{Ant-Bakht-Leb}, \cite[Remark 1.10]{kwa-trans}, \cite{kwa-interact}. We note that in general the theories of Exel's crossed products and crossed products considered in this paper are different. As explained in  \cite{kwa-exel} they can be  naturally unified  in the realm of crossed products by completely positive maps.

Our strategy to  investigate  $C^*(A,\alpha,J)$ is similar, but yet slightly different,   to the very popular method of  dilations and corners, see \cite{Laca} as well as \cite{cuntz}, \cite{Paschke}, \cite{cuntz2}, \cite{Stacey}, \cite{Rordam}. The latter method works only for injective endomorphisms and allows one to  regard  crossed products by  endomorphisms as   (full) corners in  crossed products by automorphisms. Therefore it gives a realization of the crossed product only up to Morita equivalence. This might  be too insensitive tool when studying more  subtle problems such as, for instance, spectra of   elements of the algebras under consideration.  

In our approach we start with an arbitrary $C^*$-dynamical system $(A,\alpha)$ and an ideal $J$ in $(\ker\alpha)^\bot$.  We construct a canonical reversible $C^*$-dynamical system $(B,\beta)$ that extends $(A,\alpha)$. We call $(B,\beta)$ a \emph{natural reversible $J$-extension} of $(A,\alpha)$. It is defined as a direct limit in the category of $C^*$-dynamical systems, and identified as a universal object in the sense to be specified below. Examples of such objects in the commutative case were studied in \cite{maxid}, \cite{kwa-logist}, and when $A$ is unital  were constructed in \cite{kwa-ext}. In particular, it follows that we have a natural isomorphism 
$$
C^*(A,\alpha,J)\cong C^*(B,\beta).
$$
Our main goal is to use the above isomorphism to get a description of the ideal structure of $C^*(A,\alpha,J)$. The tactics  is as follows. Firstly, we  develop or extend to non-unital case  the tools that apply to $C^*(B,\beta)$, cf. \cite{kwa-interact}, \cite{Ant-Bakht-Leb}. Secondly, by an analysis of the relationship between $(A,\alpha)$ and $(B,\beta)$ we derive results for general crossed products $C^*(A,\alpha,J)$. The general relationship,  established in this way, between the lattices of ideals of various type  is presented by the diagram on Figure  \ref{lattice diagram}; here the arrow $A  \Longrightarrow B$ indicates that there is an order retraction\footnote{an order preserving surjection $r:A\to B$ which has an order  preserving right inverse} from the lattice $A$ onto the lattice $B$, and $A \Longleftrightarrow B$ means that $A$ and $B$ are order isomorphic  (precise descriptions will be given below).

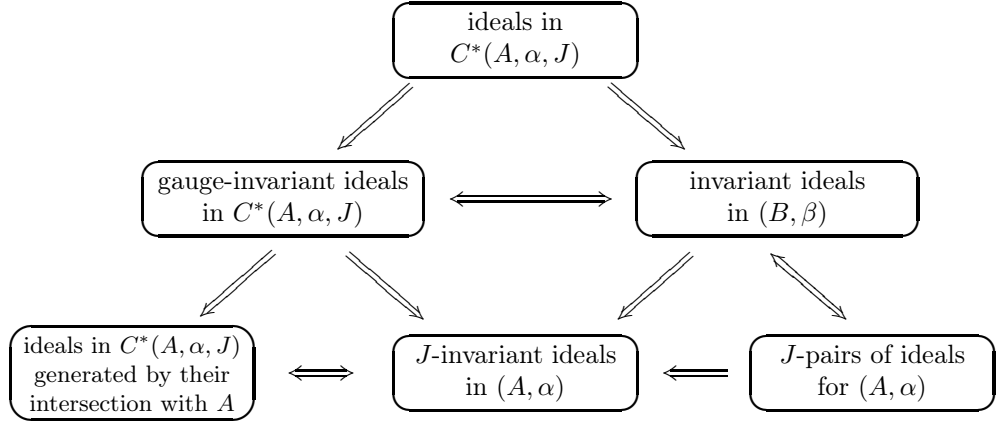
\begin{figure}[htb]
\begin{center} \begin{picture}(220,156)(0,38)
\small
 \put(65,180){\Ovalbox{\begin{minipage}{3cm}\begin{center}
ideals in $C^*(A,\alpha, J)$  
\end{center}
\end{minipage}} 
}

%\put(30,120){$\Longleftarrow$ }   
 
 \put(-30,120){\Ovalbox{\begin{minipage}{3.5cm}\begin{center}
gauge-invariant ideals \\
in $C^*(A,\alpha, J)$  
\end{center}
\end{minipage}}}

\put(37,167){$\xymatrix{ \,\,  & \,\, \ar@{=>}[dl]   \\      \,\,\,  & \,\,   }$ } 

\put(137,167){$\xymatrix{ \,\,   \ar@{=>}[dr]  & \,\,\, \\      \,\,\,  & \,\,   }$ } 
 \put(76,120){$\xymatrix{ \,\, \, \ar@{<=>}[rr] & \,\,\,\, & \,\,     }$ }   
\put(157,120){\Ovalbox{\begin{minipage}{3.5cm}\begin{center}
invariant ideals \\in $(B,\beta)$  
\end{center}
\end{minipage}}}

\put(-15,104){$\xymatrix{ \,\,  & \,\, \ar@{=>}[dl]   \\      \,\,\,  & \,\,   }$ } 
\put(143,103){$\xymatrix{ \,\,  & \,\, \ar@{=>}[dl]   \\      \,\,\,  & \,\,   }$ } 
\put(38,103){$\xymatrix{ \,\,   \ar@{=>}[dr]  & \,\,\, \\      \,\,\,  & \,\,   }$ } 
\put(198,103){$\xymatrix{ \,\,   \ar@{<=>}[dr]  & \,\,\, \\      \,\,\,  & \,\,   }$ } 

\put(15,54){$\xymatrix{ \,\, \, \ar@{<=>}[r] & \, \qquad \qquad  &\,\,   }$ }

\put(115,54){$\xymatrix{ \,\, \,  \, \qquad \qquad  &  \ar@{=>}[l]    }$ } 
 
\put(65,54){\Ovalbox{\begin{minipage}{3cm}\begin{center}
$J$-invariant ideals 
\\in $(A,\alpha)$  
\end{center}
\end{minipage}}} 

\put(-80,54){\Ovalbox{\begin{minipage}{3cm}\begin{center}
 \footnotesize ideals in $C^*(A,\alpha, J)$
generated by their intersection with $A$  
\end{center}
\end{minipage}}}

\put(200,54){\Ovalbox{\begin{minipage}{3cm}\begin{center}
$J$-pairs of ideals
\\
for  $(A,\alpha)$  
\end{center}
\end{minipage}}}
  \end{picture} \end{center}
  \caption{Relationships between the lattices of ideals\label{lattice diagram}}
 \end{figure}
For reversible $C^*$-dynamical systems  the theory looks quite similar to the  classic theory for automorphic actions. For instance,  adopting parts of \cite{kwa-interact} to non-unital case, we  identify a condition on $(B,\beta)$ under which all  ideals in $C^*(A,\alpha, J)$ are gauge-invariant (see Corollary \ref{freeness implies gauge-invariance} below).  One of the difficult problems is to phrase such conditions in terms of $(A,\alpha)$ and $J$, which we do in full generality in the commutative case (Proposition \ref{proposition on ideals in commutative}). Moreover,  as  it can be well seen  from our analysis,  an essential difference comparing to  the classic theory  arise for systems $(A,\alpha)$ where the kernel $\ker\alpha$ is not a complemented ideal in $A$.  For such systems, even in the unrelative crossed product $C^*(A,\alpha)$, not all gauge-invariant ideals are generated by their intersection with $A$. In order to describe  such ideals in general (Theorem \ref{lattices descriptions main thm}) we need  a pair of ideals in $A$.  Similar facts were first observed in  the context of graph $C^*$-algebras \cite{bhrs} and  for general relative Cuntz-Pimsner are described in \cite{ka3}.   

Another new phenomenon occurs in the study of simplicity of $C^*(A,\alpha,J)$ (Theorem \ref{theorem on simplicity}). We show that  $C^*(A,\alpha,J)$ is not simple unless it is the  unrelative crossed product $C^*(A,\alpha)$ and $(A,\alpha)$ is minimal. The minimality condition we use (Definition \ref{minimality definitions})   is slightly more subtle then the corresponding condition used in the unital case, see \cite{Paschke}, \cite{Murphy}, \cite{Murphy2}, \cite{Szwajcar}. Moreover, the simplicity of $C^*(A,\alpha)$ implies the following dynamical dichotomy:  either
\begin{itemize}
\item[1)] $\alpha$ is pointwise quasinilpotent, or 
\item[2)] $\alpha$ is a monomorphism and no power $\alpha^n$, $n>0$, is inner.
\end{itemize}
Conversely, for a minimal $(A,\alpha)$ we show that 1) is  sufficient for simplicity of $C^*(A,\alpha)$, and if $A$ is unital we infer from  \cite[Theorem 4.1]{Szwajcar}  that 2) implies  $C^*(A,\alpha)$ is  simple.  We also examine condition 2) in the non-unital case, for reversible systems and arbitrary systems on commutative algebras (see respectively Corollary \ref{cor for simplicity} and Theorem \ref{simplicity in commutative case}). We remark that condition 2) is related to the classic theorem of Kishimoto \cite{Kishimoto} while 1) can not occur for non-injective $\alpha$.   

It is fair to say that we could analyze $C^*(A,\alpha,J)$   as a relative Cuntz-Pimsner algebra (or even as a $C^*$-algebra of an ideal in a right tensor $C^*$-precategory, cf. \cite[Example 2.20]{kwa-doplicher}).  Such  analysis would heavily rely on the results of \cite{ka3}; we  briefly explain it in  the appendix section. However, our direct approach has several advantages. It makes the content   accessible for the broader audience and allows the presentation to be self-contained.  
Hopefully, more direct  arguments   shed  more light on the investigated problems and open the door to further future generalizations, for instance, in the purely algebraic setting, cf. \cite{AGBGP}.  Last but not least, we think that the method of using  reversible extensions developed in the paper has a potential. For example, we plan to use it  in forthcoming papers to tackle endomorphisms of $C_0(X)$-algebras and to analyze spectra of the related weighted composition operators, cf. \cite{Anton_Lebed}.

The content of the paper is essentially divided into two parts. %is organized as follows.

Sections \ref{endomorphism section} and \ref{reversible extension subsection}  are `space free' - all considerations are performed in a category of $C^*$-dynamical systems  \dyna. In section \ref{endomorphism section} we discus reversible $C^*$-dynamical systems which are special objects in \dyna.   We also  introduce and study the notion of  $J$-covariance for morphisms in \dyna.  In section \ref{reversible extension subsection} we construct \emph{the natural reversible $J$-extension} $(B,\beta)$ of a  $C^*$-dynamical system $(A,\alpha)$ as a certain direct limit in  \dyna, and  characterize it as a universal object (Theorem \ref{granica prosta zepsute twierdzenie}). Invariant ideals and $J$-pairs for $(A,\alpha)$ are introduced  in subsection  \ref{Invariant ideals and J-pairs subsection}. The important tool is Theorem \ref{invariant ideals and J-pairs thm} where we describe   invariant ideals in $(B,\beta)$, and the corresponding quotient subsystems, in terms of  $J$-pairs for $(A,\alpha)$. 

In Section \ref{crossed products section} we investigate representation theory for $C^*$-dynamical systems.  In subsection \ref{covariant subsection} we generalize the notion of a $J$-covariant representation of $(A,\alpha)$, \cite[Definition 1.7]{kwa-leb}, and show that  such representations can  be treated as   $J$-covariant morphisms from $(A,\alpha)$ to reversible $C^*$-dynamical systems. This allows us to apply the previously elaborated tools to study the crossed product $C^*(A,\alpha,J)$, which we define in subsection \ref{subsection crossed products}. In particular, we  describe the lattice of its gauge-invariant ideals in Theorem \ref{lattices descriptions main thm} and establish  the simplicity criteria in Theorem \ref{theorem on simplicity}. In subsection \ref{Topological freeness} we exploit the notion of topological freeness defined for reversible $C^*$-dynamical systems. Finally, we illustrate the theory with a discussion of the commutative case, see subsection \ref{Commutative $C^*$-dynamical systems}.  

 We close the paper with an appendix indicating how to analyze $C^*(A,\alpha,J)$ using the general theory of relative Cuntz-Pimsner algebras. We focus here  mainly on the ideal structure, but this also  leads to  criteria for other properties of $C^*(A,\alpha,J)$ such as nuclearity or exactness.

\subsection{Preliminaries and notation}
  All  ideals in $C^*$-algebras are  assumed to be closed and two sided.  We denote by
  $
  I^\bot=\{a\in A:aI=\{0\}\}$  the \emph{annihilator} of an ideal $I$ in a $C^*$-algebra $A$ (it is the largest ideal in $A$ with the property that $I^\bot\cap I=\{0\}$).  An ideal $I$ in $A$ is called \emph{essential} if $I^\bot=\{0\}$. A \emph{multiplier algebra} $M(A)$ of $A$ can be characterized as  a maximal unital $C^*$-algebra containing $A$ as an essential ideal. We will denote the unit in $M(A)$ simply by $1$. Occasionally we will use the same symbol for units in different algebras but this should not cause  confusion.
  One of the models for $M(A)$ is  where the elements of $M(A)$ are viewed as functions on $A$ possessing (necessarily unique) adjoints:
  $$
  M(A)=\{m:A\to A: \exists_{m^*:A\to A} \,\, (ma)^*b=a^* (m^*b) \textrm{ for all } a,b\in A \}.
  $$ 
 Another approach is to view $M(A)$ as a completion of $A$ in the strict topology. A net $m_\lambda$ in $M(A)$ converges strictly to $m\in M(A)$ if and only if $m_\lambda a \to ma$ and $m_\lambda^* a \to m^*a$ in $A$, for any $a\in A$. Since  we will need notation only for closures in the strict topology we adopt the convention that   for any  subset  $I$ of $A$
 $$
 \overline{I}\textrm{ stands for the \emph{strict closure} of } I\subseteq A \textrm{ in } M(A) .
 $$
   By  homomorphisms, epimorphisms, etc. between $C^*$-algebras we always mean  $*$-preserving maps. A continuous mapping $T:A\to B$ between $C^*$-algebras $A$ and $B$ is said to be \emph{extendible}, cf. \cite{adji}, if it extends to a (necessarily unique) strictly continuous mapping $\overline{T}:M(A)\to M(B)$.  A homomorphism $T:A\to B$ is extendible if and only if for some (and  hence any) approximate unit $\{\mu_\lambda\}$  in $A$ the net $\{T(\mu_\lambda)\}$ converges strictly in $M(B)$.  Then  $\overline{T}$ is determined by the formula $\overline{T}(m)b=\lim T(m\mu_\lambda)b$, $m\in M(A)$, $b\in B$. 
If the net $\{T(\mu_\lambda)\}$ converges strictly to the unit in $M(B)$ homomorphism $T$ is called \emph{non-degenerate}.

The following relations between strict closures and strictly continuous extensions will be of constant use in the paper.
\begin{lem}\label{diamonds lemma}
Let $I$ be an ideal in $A$ and  $\overline{I}$ its strict closure in $M(A)$. We have
$$
\overline{I}=\{b \in M(A): bA \subseteq I\} \qquad \textrm{and} \qquad \overline{I}^\bot=\overline{I^\bot} 
$$
where $\overline{I}^\bot$ is the annihilator of $\overline{I}$ in $M(A)$ and $\overline{I^\bot}$ is the strict closure of the annihilator of $I$  in $A$. 
If $T:A\to B$ is an extendible homomorphism, then
$$
\ker \overline{T} =\overline{\ker T}, \qquad  (\ker \overline{T})^\bot =\overline{(\ker T)^\bot}.
$$
\end{lem}
\begin{proof} 
Clearly $\{b \in M(A): bA \subseteq I\}$ is strictly closed and contains $I$, whence $\overline{I}\subseteq \{b \in M(A): bA \subseteq I\}$.  Moreover, if $b \in M(A)$ is such that $bA \subseteq I$  and $\{\mu_\lambda\}$ is an approximate unit in $A$, then $\{b\mu_\lambda\}\subseteq I$ is a net  converging strictly  to $b$. Therefore  $\overline{I}=\{b \in M(A): bA \subseteq I\}$. Now let $b \in M(A)$. Using what we have just proved  we get
\begin{align*}
b \in \overline{I}^\bot &\Longleftrightarrow b\overline{I} =\{0\} \Longleftrightarrow b\overline{I} A=\{0\} \Longleftrightarrow b I=\{0\}  \Longleftrightarrow bA I=\{0\}  \Longleftrightarrow bA \subseteq I^\bot
\\
&   \Longleftrightarrow b\in \overline{I^\bot}.
\end{align*}
This shows the first part of the assertion. For the second part note that  we have $\overline{\ker T} \subseteq \ker\overline{T}$, and the reverse inclusion follows because
$$
b\in \ker\overline{T} \Longrightarrow T(bA)= \overline{T}(b)T(A)=0   \Longrightarrow bA \subseteq \ker T  \Longrightarrow b \in \overline{\ker T} .
$$
Now we get $(\ker \overline{T})^\bot =(\overline{\ker T})^\bot=\overline{(\ker T)^\bot}$.
\end{proof}
If $I$ is an ideal in $A$ we denote by $q_I:A\to A/ I$ the quotient map, and occasionally when the context is clear we suppress the subscript $I$ and write simply $q$. 
 We extend the standard notation concerning the sum of ideals in $C^*$-algebras, and if $D_0,D_2,..,D_n,...$ are linear subspaces  in a $C^*$-algebra we put
$$
\sum_{k=0}^n D_k:=\{\sum_{k=0}^{n}d_k: d_k \in D_k\}\quad \textrm{ and }\quad\sum_{k=0}^\infty D_k:=\clsp\{ d_k: d_k \in D_k, k \in \N\}.
$$
Hence $\sum_{k=0}^\infty D_k$ is the closure of the increasing sum of the spaces $\sum_{k=0}^n D_k$, $n\in \N$. 
 Similarly, we put  $CD:=\{cd: c\in C, b\in D\}$ whenever an action of $C$ on $D$ makes sense.
By  Cohen-Hewitt factorization theorem  a homomorphism $T:A\to B$ between $C^*$-algebras $A$ and $B$ is non-degenerate if and only if $T(A)B=B$.

\section{The category of $C^*$-dynamical systems}\label{endomorphism section}

We start by introducing the category of $C^*$-dynamical systems which  we will work with.
\begin{defn}
A \emph{$C^*$-dynamical system} is a pair $(A,\alpha)$ where $A$ is a $C^*$-algebra and $\alpha:A\to A$ is an extendible endomorphism.
A \emph{morphism} from a $C^*$-dynamical system $(A,\alpha)$ to a $C^*$-dynamical system $(B,\beta)$ is a non-degenerate homomorphism $T:A\to B$ such that $T\circ \alpha=\beta \circ T$. We will signalize this by writing $(A,\alpha)\stackrel{T}{\rightarrow}(B,\beta)$. The arising category will be denoted by \dyna.
\end{defn}
Passage to strictly continuous extensions yields a  functor from  $\dyna$   onto its full  subcategory whose objects are endomorphisms of unital $C^*$-algebras. 
\begin{lem}
If $(A,\alpha)\stackrel{T}{\rightarrow}(B,\beta)$ then $(M(A),\overline{\alpha})\stackrel{\overline{T}}{\rightarrow}(M(B),\overline{\beta})$.
\end{lem}
\begin{proof}
Clear by strict density of $A$ in $M(A)$ and strict continuity of $\overline{\alpha}$, $\overline{\beta}$ and  $\overline{T}$.
\end{proof}
\subsection{Reversible $C^*$-dynamical system }\label{subsection reversible systems}
We distinguish certain objects in $\dyna$ which will view as  reversible systems. In the commutative case they correspond to partial homeomorphisms, see \cite{maxid}, \cite{kwa-logist},  and subsection \ref{Commutative $C^*$-dynamical systems} below. In noncommutative setting  the role of an inverse to an endomorphism plays  a complete transfer operator. 

Transfer operators  \cite{exel2} and  complete transfer operators \cite{Ant-Bakht-Leb} were originally introduced for unital $C^*$-algebras. Nevertheless, they can be easily adapted to our context. Let  $(A,\alpha)$ be a $C^*$-dynamical system. A {\em transfer
operator} for $(A,\alpha)$    is a positive linear map $L:A\to A$ 
 such that
\begin{equation}
L(\alpha(a)b) =aL(b),\quad \textrm{ for all }a,b\in A.
\label{b,,2}
\end{equation}
The transfer operator $L$ is  a \emph{complete transfer operator} if it satisfies 
\begin{equation}\label{definition of complete}
\alpha(L(a))=\overline{\alpha}(1)a \overline{\alpha}(1), \quad\textrm{ for all } a\in A.
\end{equation}
The following statement could be deduced from the results of \cite{kwa-exel}, cf. also \cite{kwa-trans}. For the sake of completeness we give an independent  proof. 

\begin{prop}\label{Proposition on transfers}
A $C^*$-dynamical system  $(A,\alpha)$ admits a complete transfer operator if and only if the range $\alpha(A)$ is a hereditary subalgebra of $A$ (hence a corner)  and the kernel $\ker\alpha$ is a complemented ideal in $A$.  

Moreover, if it exists, the complete transfer operator for $(A,\alpha)$ is unique and  is given by the formula
\begin{equation}\label{complete transfer operator form}
\alpha_*(a)=\alpha^{-1}(\overline{\alpha}(1)a\overline{\alpha}(1)), \qquad a\in A, 
\end{equation}
where $\alpha^{-1}$ is the inverse to the isomorphism $\alpha: (\ker\alpha)^\bot \to \alpha(A)=\overline{\alpha}(1)A\overline{\alpha}(1)$.
\end{prop}
\begin{proof}
Suppose that $\alpha_*$ is a complete transfer operator for $(A,\alpha)$. It follows from \eqref{definition of complete} that $\alpha(A)=\overline{\alpha}(1)A \overline{\alpha}(1)$  and  $\alpha \circ \alpha_*\circ \alpha=\alpha$. Hence $\alpha(A)$ is a hereditary subalgebra of $A$ and the map $p:=\alpha_* \circ\alpha$ is an idempotent acting on $A$.  Relations $\alpha=\alpha\circ p$ and $p=\alpha_* \circ\alpha$ imply that $\ker p=\ker\alpha$. Moreover,  \eqref{b,,2} implies that $\alpha_*(A) \subseteq (\ker\alpha)^\bot$ and hence $p(A) \subseteq (\ker\alpha)^\bot$. Since $p$ is an idempotent this shows that   $p(A)=(\ker\alpha)^\bot$.  As a consequence we also get $\alpha_*(A)=(\ker\alpha)^\bot$. 
To prove that $\ker\alpha$ is a complemented ideal it suffices to show that $p$ is multiplicative. Since $\alpha$ is injective on  $p(A)=(\ker\alpha)^\bot$, this follows from the calculation
$$
\alpha\left(p(a \cdot b)\right)=\alpha(a \cdot b)=\alpha(a)\alpha(b)= \alpha(p(a)) \alpha(p(b))= \alpha\left(p(a)\cdot p(b)\right)
$$
 where $a,b \in A$.  In particular, $\alpha: (\ker\alpha)^\bot \to \overline{\alpha}(1)A\overline{\alpha}(1)$ is an isomorphism. Taking into account \eqref{definition of complete} and $\alpha_*(A)=(\ker\alpha)^\bot$ we get \eqref{complete transfer operator form}.

Conversely, if $\alpha(A)$ is hereditary in $A$ it is automatically the corner $\overline{\alpha}(1)A \overline{\alpha}(1)$. If additionally the ideal $\ker\alpha$ is complemented then   $\alpha: (\ker\alpha)^\bot \to \overline{\alpha}(1)A\overline{\alpha}(1)$ is an isomorphism. In this case one easily checks that \eqref{complete transfer operator form} defines a complete transfer operator.
\end{proof}
\begin{defn} If a  $C^*$-dynamical system  $(A,\al)$ possess the (necessarily unique) complete transfer operator we  denote it by  $\al_*$ and we call $(A,\al)$  a \emph{reversible $C^*$-dynamical system}. 
\end{defn}
\begin{rem}
It readily follows from  \eqref{complete transfer operator form} that the  complete transfer operator $\alpha_*$ is a generalized inverse to $\alpha$, that is we have  $\alpha=\alpha\circ \alpha_*\circ \alpha$ and $\alpha_*=\alpha_*\circ \alpha\circ \alpha_*$.
Moreover, one can show,  cf. \cite{kwa-exel}, that if the complete transfer operator exists it is the unique transfer operator satisfying $\alpha=\alpha\circ \alpha_*\circ \alpha$.
\end{rem}
A complete transfer operator always admits a strictly continuous extension and in particular is a transfer operator in the sense of \cite{brv}, cf.     \cite{kwa-exel}. 
\begin{prop}\label{lemma for extensions of transfers}
A  $C^*$-dynamical system $(A,\alpha)$ is  reversible if and only if the  extended system $(M(A),\overline{\alpha})$ is reversible.

If $(A,\alpha)$ is reversible, the complete transfer operator $\overline{\alpha}_*$ for $(M(A),\overline{\alpha})$ is a strict continuous extension of $\alpha_*$, and  $\{\overline{\alpha}_*^n(1)\}_{n\in \N}$ is  a decreasing   sequence of central projections in $M(A)$ with $\overline{\alpha}_*^n(1)A=\ker(\alpha^n)^\bot$. 
\end{prop}
\begin{proof} 
By Lemma \ref{diamonds lemma},  $\ker\overline{\alpha}=\overline{\ker\alpha}$. Thus $\ker\alpha$ is complemented in $A$ if and only if $\ker\overline{\alpha}$ is complemented in $M(A)$.
Now, if $(M(A),\overline{\alpha})$ is reversible then   $\alpha(A)=\overline{\alpha}(A)=\overline{\alpha}(\overline{\alpha}_*(A))=\overline{\alpha}(1)A\overline{\alpha}(1)$ is hereditary in $A$, and hence $(A,\alpha)$ is reversible by Proposition \ref{Proposition on transfers}.
  Conversely, if  $(A,\alpha)$ is  reversible then  for any net $\{a_\lambda\} \subseteq A$ strictly convergent to $a\in M(A)$ the net $\{\overline{\alpha}(1) a_\lambda \overline{\alpha}(1)\}=\{\alpha(\alpha_*(a_\lambda))\}\subseteq \overline{\alpha}(M(A))$ is convergent to $\overline{\alpha}(1)a \overline{\alpha}(1)$.  It follows that  $\overline{\alpha}(M(A))=\overline{\alpha}(1)M(A)\overline{\alpha}(1)$   is hereditary  in $M(A)$ and hence $(M(A),\overline{\alpha})$ is reversible by Proposition \ref{Proposition on transfers}. 

Suppose now that the systems $(A,\alpha)$  and   $(M(A),\overline{\alpha})$ are reversible. Then  $\overline{\alpha}_*$ is a strictly continuous extension of $\alpha_*$ because it is given by   \eqref{complete transfer operator form} where $\alpha^{-1}$ is replaced by the the inverse to the strictly continuous isomorphism $\overline{\alpha}: \overline{(\ker\alpha)^\bot}=(\ker\overline{\alpha})^\bot \to \overline{\alpha(A)}=\overline{\alpha}(1)M(A)\overline{\alpha}(1)$.
Furthermore,  for each $n\in \N$, $\alpha_*^n$ is the  complete transfer operator for  $(A,\alpha^n)$.  In particular,  $\overline{\alpha}_*^n(1)$ is a projection onto $\ker(\overline{\alpha}^n)^\bot$.
\end{proof}
\subsection{Extensions and covariant morphisms}
\begin{defn} If $T$ is an injective morphism   $(A,\alpha)\stackrel{T}{\rightarrow}(B,\beta)$ we call it an  \emph{embedding} of $(A,\al)$ into $(B,\beta)$, and say  that  $(B,\beta)$ is an \emph{extension} of  $(A,\al)$. If  $T$ is an isomorphism then $(A,\al)$ and $(B,\beta)$ are \emph{equivalent} (in the category $\dyna$).  
\end{defn}
Note that if $T$ is an embedding  of $(A,\al)$ to $(B,\beta)$ then 
$$
\ker\al =T^{-1}(\ker\beta), \qquad (\ker\al)^\bot \supseteq T^{-1}((\ker\beta)^\bot).
$$
 These relations could be interpreted as follows: the extended endomorphism $\beta$ may enlarge the kernel of  $\al$ but only outside $A$, and if it does, it may shrink the annihilator of the kernel inside $A$. Thus  the ideal $J=T^{-1}((\ker\beta)^\bot)$  is a parameter 
 that measures how far can $\ker\beta$ from $\ker\al$ go. We have $\{0\}\subseteq J \subseteq (\ker\al)^\bot$ and  when $J=(\ker\al)^\bot$ the relationship between the kernels is  the strongest. This leads us to the following definition.

\begin{defn}\label{12345}
Let   $J$ be an ideal in $(\ker\al)^\bot$. We will say that $(B,\beta)$ is a \emph{$J$-covariant  extension of $(A,\al)$}, or briefly a  $J$-\emph{extension of} $(A,\al)$, if  there exits an embedding $(A,\alpha)\stackrel{T}{\rightarrow}(B,\beta)$ such that
\begin{equation}\label{J-covariance equality}
J= T^{-1}( (\ker\beta)^\bot).
\end{equation}
In this event  $T$ will be called a \emph{$J$-covariant embedding}, or briefly \emph{$J$-embedding}, of  $(A,\al)$ into $(B,\beta)$. If  $J=(\ker\al)^\bot$  instead of $J$-covariant extension and $J$-covariant embedding we will  use    terms  \emph{covariant extension} and \emph{covariant embedding}. 
\end{defn}
It is useful to extend the above notion of covariance to (not necessarily injective) morphisms  $(A,\alpha)\stackrel{T}{\rightarrow}(B,\beta)$. But, in general, since $\ker T\subseteq T^{-1}( (\ker\beta)^\bot)$   one can not expect  to have equality \eqref{J-covariance equality} for an ideal $J$ in $(\ker\al)^\bot$. One of possible  ways  to circumvent this problem is to replace \eqref{J-covariance equality}  by an inclusion.

\begin{defn}\label{general def of covariance} Let   $J$ be an ideal in $(\ker\al)^\bot$. We say that a  morphism   $(A,\alpha)\stackrel{T}{\rightarrow}(B,\beta)$  is  \emph{covariant on $J$} if 
$$
J\subseteq T^{-1}( (\ker\beta)^\bot).
$$
We say  $T$ is \emph{covariant} if it is covariant on $(\ker\alpha)^\bot$.
\end{defn}
\begin{rem}
An embedding $T$ is $J$-covariant if and only if  $J$ is the largest ideal on which $T$ is covariant as a morphism. In particular, $T$ is a covariant embedding if and only if  $T$ is an injective covariant morphism. 
\end{rem}

We  show now certain properties of extensions and morphisms we will  use latter. They already indicate a special role played by covariant morphisms.
\begin{lem}\label{covariance for extensions}
If $(A,\alpha)\stackrel{T}{\rightarrow}(B,\beta)$ is morphism, then
\begin{equation}\label{strict closures and pre-images}
\overline{T^{-1}((\ker\beta)^\bot)}=\overline{T}^{-1}((\ker\overline{\beta})^\bot), 
\end{equation}
recall that the bar on the left-hand side denotes the strict closure.
In particular,
\begin{itemize}
\item[i)] $(A,\alpha)\stackrel{T}{\rightarrow}(B,\beta)$ is covariant on $J \subseteq (\ker\alpha)^\bot$  iff $(M(A),\overline{\alpha})\stackrel{\overline{T}}{\rightarrow}(M(B),\overline{\beta})$ is covariant on  $\overline{J} \subseteq (\ker\overline{\alpha})^\bot$,
\item[ii)] $(A,\alpha)\stackrel{T}{\rightarrow}(B,\beta)$ is  $J$-covariant embedding   iff $(M(A),\overline{\alpha})\stackrel{\overline{T}}{\rightarrow}(M(B),\overline{\beta})$ is  $\overline{J}$-covariant embedding,
\item[iii)]  $(A,\alpha)\stackrel{T}{\rightarrow}(B,\beta)$ is a covariant morphism iff  $(M(A),\overline{\alpha})\stackrel{\overline{T}}{\rightarrow}(M(B),\overline{\beta})$ is a covariant morphism.
\end{itemize}
\end{lem}
\begin{proof}
By Lemma \ref{diamonds lemma} we have $(\ker\overline{\beta})^\bot=\overline{(\ker\beta)^\bot}$. Hence  
$$T^{-1}((\ker\beta)^\bot) \subseteq \overline{T}^{-1}((\overline{\ker\beta)^\bot}) \,\,\Longrightarrow  \,\,\overline{T^{-1}((\ker\beta)^\bot)}\subseteq \overline{T}^{-1}( (\ker\overline{\beta})^\bot).
$$ 
For the reverse inclusion note that for any  $b \in  \overline{T}^{-1}( (\ker\overline{\beta})^\bot)$ we have 
$$
T(bA)=\overline{T}(b)T(A)\subseteq  (\ker\overline{\beta})^\bot \cap B=\overline{(\ker\beta)^\bot} \cap B=(\ker\beta)^\bot.
$$
Hence $bA\subseteq T^{-1}((\ker\beta)^\bot)$ and therefore $b\in \overline{T^{-1}((\ker\beta)^\bot)}$ by Lemma \ref{diamonds lemma}. 
This proves \eqref{strict closures and pre-images}. The  second part of the assertion is now straightforward. 
\end{proof}

\begin{lem}\label{covariance in terms of projection}
Let  $(A,\alpha)\stackrel{T}{\rightarrow}(B,\beta)$ and suppose that  $\ker\alpha$ and $\ker \beta$ are complemented ideals in $A$ and $B$, respectively. The following conditions are equivalent:
\begin{itemize}
\item[i)] $T$ is covariant,
\item[ii)] $\overline{T}(1_{(\ker\alpha)^\bot})\leq 1_{(\ker\beta)^\bot}$,
\item[iii)] $\overline{T}(1_{(\ker\alpha)^\bot})=1_{(\ker\beta)^\bot}$,
\end{itemize}
where $1_{(\ker\alpha)^\bot}\in M(A)$ and $1_{(\ker\beta)^\bot}\in M(B)$ denote the projections onto $(\ker\alpha)^\bot$ and $(\ker\beta)^\bot$, respectively. \end{lem}
\begin{proof}
 By passing to  extended systems, see  Lemma  \ref{covariance for extensions} iii), we may assume that both $A$ and $B$ are unital. Then $T$ is  unital   and   $1_{(\ker\alpha)^\bot}$ and $1_{(\ker\beta)^\bot}$ are units in  $(\ker\alpha)^\bot$ and $(\ker\beta)^\bot$, respectively. Now, one readily sees that the inclusion $(\ker\alpha)^\bot \subseteq T^{-1}((\ker\beta)^\bot)$ is equivalent to the inequality $T(1_{(\ker\alpha)^\bot})\leq 1_{(\ker\beta)^\bot}$. Hence i) is equivalent to ii). Moreover, $T\circ \alpha=\beta \circ T$ implies that $T(1_{\ker\alpha})\leq 1_{\ker\beta}$ where $1_{\ker\alpha}$ and $1_{\ker\beta}$ are units in $\ker\alpha$ and $\ker\beta$, respectively. Thus, as $T$ is unital,  the inequality $T(1_{(\ker\alpha)^\bot})\geq 1_{(\ker\beta)^\bot}$ is always satisfied. Hence ii) is equivalent to iii).
\end{proof}
In connection with the following statements  recall, see Proposition \ref{lemma for extensions of transfers}, that if $\alpha_*$ is a complete transfer operator for $(A,\alpha)$, then $\overline{\al}_*(1)$ is the unit in $(\ker\overline{\alpha})^\bot=\overline{(\ker\alpha)^\bot}$. %, Lemma \ref{diamonds lemma}.
 \begin{prop}\label{lemma for reversibles} Suppose   $(A,\al)$ and $(B,\beta)$ are reversible $C^*$-dynamical systems. For any  morphism $(A,\alpha)\stackrel{T}{\rightarrow}(B,\beta)$ the following conditions are equivalent:
 \begin{itemize}
 \item[i)]  $T$  is covariant,
 \item[ii)]  $\beta_*\circ T=T\circ \al_*$, 
 \item[iii)] $\overline{T}(\overline{\al}_*(1)) \leq \overline{\beta}_*(1)$,
 \item[iv)] $\overline{T}(\overline{\al}_*(1))=\overline{\beta}_*(1)$.
       \end{itemize}
      \end{prop}
       \begin{proof} 
Equivalences i)$\Leftrightarrow$ iii) $\Leftrightarrow$ iv) follow from Lemma \ref{covariance in terms of projection}. The implication  ii)$\Rightarrow$ iv) is immediate. To close the cycle we show that  i)$\Rightarrow$ ii).  To this end recall that the mappings $A\ni a \mapsto (\alpha_*\circ\alpha)(a)=\overline{\alpha}_*(1)a \in (\ker\alpha)^\bot$ and $A\ni a \mapsto (\alpha\circ\alpha_*)(a)=\overline{\alpha}(1)a\overline{\alpha}(1) \in \overline{\alpha}(1)A\overline{\alpha}(1)$ are  projections. Similarly, for $(B,\beta)$. Now using covariance of $T$ we get $T(\al_*(a))\in (\ker\beta)^\bot$, for any $a\in A$. Therefore  
\begin{align*}
T(\alpha_*(a))&=\beta_*\big(\beta(T(\al_*(a)))\big)=\beta_*\big(T (\al(\al_*(a)))\big)=\beta_*\big(T(\overline{\al}(1)a\overline{\al}(1))\big)
\\
&=\beta_*(\overline{\beta}(1)T(a)\overline{\beta}(1))=\beta_*(T(a)). 
 \end{align*}  \end{proof}  
 \begin{lem}\label{lemma describing ideal J}
Suppose that $(A,\alpha)\stackrel{T}{\rightarrow}(B,\beta)$ is an embedding and  $(B,\beta)$ is a reversible $C^*$-dynamical system. Then $T$ is $J$-covariant where
$$
J=\{a\in A : \, \overline{\beta}_*(1) T(a)=T(a)\}.
$$
\end{lem}
\begin{proof}
We have 
$\{a\in A: \overline{\beta}_*(1)T(a)=T(a)\}=\{a\in A: T(a)\in(\ker\beta)^\bot\}=T^{-1}((\ker\beta)^\bot)$.
\end{proof}
 
Composition with covariant embeddings does not affect the covariance.
 \begin{lem}\label{composition of embeddings}
 Let $(B,\beta)\stackrel{S}{\rightarrow}(C,\gamma)$ be a covariant embedding. A morphism  $(A,\alpha)\stackrel{T}{\rightarrow}(B,\beta)$ is covariant on  $J$ iff  $(A,\alpha)\stackrel{S\circ T}{\longrightarrow}(C,\gamma)$ is covariant on  $J$.
      \end{lem}
       \begin{proof} 
We have    $(S\circ T)^{-1}( (\ker\gamma)^\bot)=T^{-1}( S^{-1}(\ker\gamma)^\bot)=T^{-1}( (\ker\beta)^\bot)$.
   \end{proof}
   
   We finish this subsection by noting that  $\dyna$ is a category with direct limits.
\begin{prop}\label{existence of direct limits}
A direct sequence 
$$
%\begin{equation}\label{ciag prosty C*-algebr}
(B_0,\beta_0) \stackrel{T_0}{\longrightarrow} (B_1,\beta_1)
\stackrel{T_1}{\longrightarrow} (B_2,\beta_2)  \stackrel{T_2}{\longrightarrow}...\,
%\end{equation}
$$
has   a direct limit $(B,\beta)$ in the category $\dyna$, where $B=\underrightarrow{\lim\,\,}\{B_{n},T_{n}\}$ is the $C^*$-algebraic direct limit and $\beta$ is the induced endomorphism.

Moreover, if all of the bonding maps $T_n$, $n\in \N$, are covariant morphism then the natural morphisms $(B_n,\beta_n)\stackrel{\phi_n}{\rightarrow}(B,\beta)$ are also covariant.
\end{prop}
\begin{proof}
Let $n\in \N$. Notice first that   the natural homomorphism $\phi_n:B_n \to B$ is  non-degenerate.  Indeed, any element in $B$ can be approximated by  $\phi_m(b_m)\in B$ for certain $b_m\in B_m$ and $m>n$.  As a composition of finite number of non-degenerate homomorphisms, the homomorphism $T_{m,n}:=T_m\circ ... \circ T_{n+1}\circ T_n:B_n \to B_m$ is non-degenerate. Thus   $b_m=T_{m,n}(a_n)c_m$ for certain $a_n\in B_n$, $c_m\in B_m$. Hence  $\phi_m(b_m)=\phi_n(a_n)\phi_m(c_m)$ and this proves our claim.
\\
Clearly, the formula $\beta(\phi_n(b)):=\phi_{n}(\beta_n(b))$, $b\in B_n$, $n\in \N$, yields a well defined endomorphism $\beta:B\to B$. To see it  is extendible let $\{\mu_\lambda\}$  be an approximate unit in $B_0$. As $\phi_0$ is non-degenerate,   $\{\phi_0(\mu_\lambda)\}$ is an approximate unit in $B$. 
 It follows
that the  (bounded and self-adjoint) net $\{\beta(\phi_0(\mu_\lambda))\}$ converges strictly in  $M(B)$ because  for any $b\in B_n$, $n\in \N$, the net 
$$
\beta(\phi_0(\mu_\lambda))\phi_n(b)=\phi_n(T_{n,0}(\beta_0(\mu_\lambda))b)
$$
converges to $\phi_n(\overline{T}_{n,0}(\overline{\beta}_0(1))b)$. Hence $\beta$  is extendible and $(B,\beta) \in \dyna$.
\\
Let  $(C,\gamma)$ be an arbitrary  $C^*$-dynamical system equipped with morphisms $(B_n,\beta_n)\stackrel{\psi_n}{\rightarrow}(C,\gamma)$ such that $\psi_{n+1}\circ T_n=\psi_{n}$ for all  $n\in \N$. Then the universal property of the $C^*$-algebraic direct limit $B$ implies that there is a homomorphism $\Psi:B\to C$ such that  
$
\Psi\circ \phi_n=\psi_n
$, $n\in \N$. In particular, $\Psi$ is non-degenerate because each of $\psi_n$ is. Since
$$
(\Psi\circ \beta)\circ \phi_n=\Psi\circ \phi_n\circ \beta_n=\psi_n\circ \beta_n=\gamma \circ \psi_n=(\gamma \circ \Psi)\circ \phi_n,
$$ 
we conclude that  $(B,\beta)\stackrel{\Psi}{\rightarrow}(C,\gamma)$ is a morphism, and  $(B,\beta)$ is the  direct limit in $\dyna$. 

Now suppose that  all of the bonding maps $T_m$, $m\in \N$, are covariant  and let $a\in (\ker\beta_n)^\bot$ for a fixed $n\in \N$. Covariance implies that for any $m>n$ we have  $T_{m,n}(a)\in (\ker\beta_m)^\bot$. Now let $\phi_k(b) \in \ker\beta$ where  $b\in B_k$ and $k>n$.  For any $m>k$ we have $T_{m,n}(a)T_{m,k}(b) \in (\ker\beta_m)^\bot$ and $\phi_m(\beta_m(T_{m,n}(a)T_{m,k}(b)))=0$. Using these  relations and the explicit  formula for norm in the direct limit $B$ we get 
\begin{align*}
\|\phi_n(a)\phi_k(b)\|
&=\lim_{m\to \infty }\|T_{m,n}(a)T_{m,k}(b)\|=\lim_{m\to \infty }\|\beta_m(T_{m,n}(a)T_{m,k}(b))\|\\
&=\lim_{m\to \infty }\|\phi_m(\beta_m(T_{m,n}(a)T_{m,k}(b)))\|=0.
\end{align*}
Hence $\phi_n(a)\phi_k(b)=0$ and this implies that $\phi_n(a) \in (\ker\beta)^\bot$. Accordingly, $(B_n,\beta_n)\stackrel{\phi_n}{\rightarrow}(B,\beta)$ is covariant.
\end{proof}
\begin{rem}\label{extension of direct limits remark}
The direct limit system $(B,\beta)$ sits  in both its strictly continuous extension $(M(B),\overline{\beta})$ and the direct limit system  
associated to the   direct sequence obtained by  strictly continuous extensions:
%$$
\begin{equation}\label{ciag prosty rozszerzonych C*-algebr}
(M(B_0),\overline{\beta}_0) \stackrel{\overline{T}_0}{\longrightarrow} (M(B_1),\overline{\beta}_1)
\stackrel{\overline{T}_1}{\longrightarrow} (M(B_2),\overline{\beta}_2)  \stackrel{\overline{T}_2}{\longrightarrow}...\, .
\end{equation}
%$$
In fact, with  obvious identifications we have
$
B \,\subseteq\, \underrightarrow{\lim\,\,}\{M(B_{n}),\overline{T}_{n}\}\, \subseteq \,M(B).
$

\end{rem}

\section{Natural reversible extensions of $C^*$-dynamical systems and invariant ideals}\label{reversible extension subsection}
We fix a $C^*$-dynamical system $(A,\al)$ and an ideal $J$ in $(\ker\al)^\bot$.
Our first aim  is to construct  a universal reversible $C^*$-dynamical system $(B,\beta)$ which is a $J$-extension of $(A,\al)$. Next we study relationship between invariant ideals in $(B,\beta)$ and $(A,\al)$.
 In the sequel, we will abbreviate the  long phrase `reversible $C^*$-dynamical system which is a $J$-extension' to the short expression `\emph{reversible $J$-extension}'.

\subsection{The main construction}
We will define the system $(B,\beta)$   as a direct limit of certain approximating $C^*$-dynamical systems $(B_n,\beta_n)$, $n\in \N$, that are constructed as follows.
 Denote by $q:A\to A/J$ the quotient map and for each  $n\in \N$
put
$$\label{algebry A_n 2}
A_n:=\overline{\al}^n(1)A\overline{\al}^n(1).
$$
Define the $C^*$-algebra $B_n$ as a direct sum of the form
$$
B_n=q(A_{0})\oplus q(A_{1})\oplus ... \oplus q(A_{n-1}) \oplus A_{n}.
$$
Let  $\beta_n:B_n \to B_n$  be  given by
$$
\beta_n(a_{0}\oplus a_{1}\oplus ... \oplus a_{n})=a_{1}\oplus a_{ 2}\oplus ... \oplus q(a_{n})\oplus \al(a_n).
$$
Clearly, $\beta_n$ is an extendible endomorphism. Thus we have a sequence $(B_n,\beta_n)$, $n\in \N$, of $C^*$-dynamical systems with $(B_0,\beta_0)= (A,\alpha)$. 

Define the bonding homomorphisms $T_n:B_n \to B_{n+1}$, $n\in \N$, as schematically
presented by the diagram
\begin{equation}\label{direct limit diagram to be dualized}
\begin{xy}
\xymatrix@C=3pt{
      **[r]  B_n \ar[d]^{T_n}& = &  q(A_{0}) \ar[d]^{id} &  \oplus & ... & \oplus &
      q(A_{n-1})\ar[d]^{id}& \oplus & A_{n} \ar[d]^{q}    \ar[rrd]^{\al}   \\
       B_{n+1} & = &  q(A_{0})& \oplus & ... &  \oplus& q(A_{n-1}) & \oplus & q(A_{n}) & \oplus  & A_{n+1}
        }
  \end{xy}
  \end{equation}
and formally given  by  the formula
$$
T_n (a_{0}\oplus ... \oplus a_{ n-1}\oplus a_{n})= a_{0}\oplus ... \oplus a_{ n-1}\oplus q(a_{n})
\oplus  \al(a_{n}),
$$
where $a_{k}\in q(A_{k})$, $k=0,...,n-1,$ and $a_{n}\in A_{n}$. Since $J\subseteq (\ker\al)^\bot$, the homomorphisms $T_n$ are  injective. 
Plainly,  
$T_n\circ \beta_n=\beta_{n+1}\circ T_n$, $n \in \N$. Hence  we have the direct sequence of embeddings:
\begin{equation}\label{ciag prosty C*-algebr}
(A,\alpha)=(B_0,\beta_0) \stackrel{T_0}{\longrightarrow} (B_1,\beta_1)
\stackrel{T_1}{\longrightarrow} (B_2,\beta_2)  \stackrel{T_2}{\longrightarrow}...\,
\end{equation}
\begin{thm}\label{granica prosta zepsute twierdzenie}
Let  $(B,\beta)$ be the direct limit  of the direct  sequence \eqref{ciag prosty C*-algebr}. Then  $(B,\beta)$ is a reversible  $J$-extension of  $(A,\al)$, with an embedding $T$, possessing the following properties: 
 \begin{itemize}
 \item[i)] $B=\sum_{n=0}^\infty \beta_*^n(T(A))$  and $J=\{a\in A: \overline{\beta}_*(1)T(a)=T(a)\}$.
\item[ii)] If  $(C,\gamma)$ is a reversible  $J$-extension  of  $(A,\al)$ and $S$ is the corresponding embedding, then there is a unique covariant embedding $\widetilde{S}$ of $(B,\beta)$ into $(C,\gamma)$ such that $S= \widetilde{S}\circ T$, i.e. the diagram
$$
\xymatrix{ & (A,\al)   \ar[ld]_T  \ar[rd]^S &\\
   (B,\beta)  \ar[rr]^{\widetilde{S}}  & &  (C,\gamma)   
 }
$$
commutes. 
\item[iii)] Any  reversible  extension of $(A,\alpha)$ that possess one  of the properties of $(B,\beta)$ described in item i) or item ii)  is equivalent to $(B,\beta)$.
\end{itemize}
\end{thm}
\begin{proof}
Let  $n>0$. We have   $(\ker\beta_n)^\bot=0 \oplus q(A_1)\oplus ... \oplus q(A_{n-1}) \oplus A_n$. Thus  by the form of the bonding map $T_n$, we see that  $(B_n,\beta_n) \stackrel{T_n}{\longrightarrow} (B_{n+1},\beta_{n+1})$ is a covariant embedding and clearly $(A,\alpha) \stackrel{T_0}{\longrightarrow} (B_1,\beta_1)$ is a $J$-embedding. Accordingly, denoting by $\phi_n:B_n\to B$ the natural embeddings into the direct limit $B$, we deduce from  Lemma \ref{composition of embeddings} and Proposition \ref{existence of direct limits},  that $T:=\phi_0$ is a $J$-embedding of $(A,\alpha)$ into  $(B,\beta)$.
\\
To prove that $(B,\beta)$ is reversible, recall that we may  treat $\underrightarrow{\lim\,\,}\{M(B_{n}),\overline{T}_{n}\}$ as a $C^*$-subalgebra of $M(B)$, cf. Remark \ref{extension of direct limits remark}. In particular,  for each $n>0$ the projection  from $B_n$ onto $(\ker\beta_n)^\bot$, which we denote by $p_n$, is an element of $M(B_n)$, and $\overline{T}_n(p_n)=p_{n+1}$. It follows that  $\overline{\phi}_n(p_n)$ 
is  the projection from $B$ onto $(\ker \beta)^\bot$. Hence $(\ker \beta)^\bot$ is a  complemented ideal in $B$. Similarly, for each $n>0$ we have 
$
 \overline{T}_{n}(\overline{\beta}_{n}(1))=\overline{\beta}_{n+1}(1)$, and therefore  $\overline{\beta}(1)=\overline{\phi}_n(\overline{\beta}_n(1))$. To see that 
   $\beta(B)$ is a hereditary subalgebra it suffices to note that $\overline{\beta}(1)B \overline{\beta}(1)\subseteq \beta(B)$. To show the latter put $E_p(a):=pap$, for any $a,p \in M(A)$. Then 
     $$
 \overline{\beta}(1)\left(\phi_n(q(a_{0})\oplus ... \oplus q(a_{n-1}) \oplus a_{n} )\right)\overline{\beta}(1)
   $$ 
   equals to 
   $$
\phi_n \left(q(E_{\overline{\alpha}(1)}(a_{0}))\oplus ... \oplus q(E_{\overline{\alpha}^n(1)}(a_{n-1})) \oplus E_{\overline{\alpha}^{n+1}(1)}(a_{n})\right),
   $$
   which in turn equals to 
   $$
\beta\left(\phi_{n+1}(0 \oplus q(E_{\overline{\alpha}(1)}(a_{0}))\oplus ... \oplus q(E_{\overline{\alpha}^n(1)}(a_{n-1})) \oplus E_{\overline{\alpha}^{n+1}(1)}(a_{n}) )\right).
$$ 
Thus $\beta(B)=\overline{\beta}(1)B \overline{\beta}(1)$ and by Proposition \ref{Proposition on transfers} the system $(B,\beta)$ is reversible. Moreover, if follows from the above calculation  and \eqref{complete transfer operator form} that the  complete transfer operator $\beta_*$ maps $\phi_n(B_n)$ into $\phi_{n+1}(B_{n+1})$ where
$$
\beta_*(\left(\phi_n(q(a_{0})\oplus ... \oplus q(a_{n-1}) \oplus a_{n} )\right))
$$
equals to
$$
\phi_{n+1}\left(0 \oplus q(E_{\overline{\alpha}(1)}(a_{0}))\oplus ... \oplus q(E_{\overline{\alpha}^n(1)}(a_{n-1})) \oplus E_{\overline{\alpha}^{n+1}(1)}(a_{n}) \right).
$$  
This proves the initial part of the assertion.
%$

i). Recall that $T=\phi_0$. Using the  above description of $\beta_*$ one checks that  for $b_k \in A_k$, $k=0,...,n$, we have
$$
(\beta_*^k\circ T)(b_k)=
 \phi_n\left( 0\oplus ... \oplus 0 \oplus q(b_k)\oplus q(\alpha(b_k)) \oplus ...  q(\alpha^{n-k-1}(b_k)) \oplus \alpha^{n-k}(a_{n})\right).
$$
Thus putting $b_k=a_k - \sum_{i=0}^{k-1}\alpha^{k-i}(a_i)$ where $a_k \in A_k$, $k=0,...,n$, and summing over $k$ we get
$$
\phi_n(q(a_{0})\oplus ... \oplus q(a_{n-1}) \oplus a_{n}) =\sum_{k=0}^n (\beta_*^k\circ T)\left(a_k - \sum_{i=0}^{k-1}\alpha^{k-i}(a_i)\right).
$$
Hence $B\subseteq \sum_{n=0}^\infty \beta_*^n(T(A))$. As the reverse inclusion is obvious we get $B=\sum_{n=0}^\infty \beta_*^n(T(A))$.
 We have $J=
\{a\in A: \overline{\beta}_*(1)T(a)=T(a)\}$ by Lemma \ref{lemma describing ideal J}.

ii). We claim that  the desired covariant embedding  $(B,\beta)  \stackrel{\widetilde{S}}{\rightarrow}   (C,\gamma)$ can be defined by the  formula 
\begin{equation}\label{formula szpiegula666}
 \qquad \S(\phi_n(q(a_{0})\oplus q(a_1)\oplus  ... \oplus a_{n}))=
 QS(a_{0})+ \gamma_*(Q S(a_{1})) ... + \gamma_*^n(S(a_{n}))
\end{equation}
where $Q:=(1-\overline{\gamma}_*(1))$ is a central  projection  in $M(C)$, see Proposition \ref{lemma for extensions of transfers}. To this end  notice  that $\overline{S}(\overline{\alpha}^k(1))=\overline{\gamma}^k(1)$, $k\in \N$, and  $J=
\{a\in A: \overline{\gamma}_*(1)S(a)=S(a)\}$, cf. Lemma  \ref{lemma describing ideal J}. Thus one  deduces that   
$$
A_k/J \ni q(a) \longmapsto  QS(a) \in   Q\overline{\gamma}^k(1) S(A)\overline{\gamma}^k(1)
$$
is a well defined  isomorphism. Moreover, composing it with the isomorphism $\gamma_*^k:\gamma^k(C)=\overline{\gamma}^k(1) C\overline{\gamma}^k(1)\to \overline{\gamma}_*^k(1)C$ we get the injective homomorphism 
\begin{equation}\label{A_k/J mapping}
A_k/J \ni q(a) \longmapsto \gamma_*^k( QS(a)) \in   \overline{\gamma}_*^k(Q) C. 
\end{equation}
Similarly, for each $n\in \N$, we get the injective homomorphism 
\begin{equation}\label{A_n mapping}
A_n \ni  a \longmapsto \gamma_*^n( S(a)) \in   \overline{\gamma}_*^n(1) C.
\end{equation}
Now,  since the projections $\{\overline{\gamma}_*^k(1)\}_{k\in \N}$ are decreasing and central  in $M(C)$ we see  that the 
projections  
$$
\overline{\gamma}_*^k(Q)=\overline{\gamma}_*^{k}(1)-\overline{\gamma}_*^{k+1}(1), \,\, k=0,...,n-1, \quad\textrm{and} \quad  \overline{\gamma}_*^n(1)
$$
 are pair-wise
orthogonal and central in $M(C)$. Hence the direct sum of mappings \eqref{A_k/J mapping} for $k=0,1...,n-1$ and the mapping \eqref{A_n mapping}, establishes an injective homomorphism from $B_n=\bigoplus_{k=0}^{n-1}q(A_k)\oplus A_n$ to $C=\bigoplus_{k=0}^{n-1} \overline{\gamma}_*^k(Q) C \oplus \overline{\gamma}_*^n(1) C$.
As $\phi_n:B_n \to B$ is injective,  we conclude that  the formula \eqref{formula szpiegula666} yields a well defined  injective homomorphism from $\phi_n(B_n)$ into $C$. 
To see that it actually defines  a homomorphism from $B$ to $C$ note that   
$$
S(a_{n})=Q S(a_{n}) +\overline{\gamma}_*(1)S(a_{n})=QS(a_{n}) +\gamma_*(\gamma(S(a_{n}))=QS(a_{n}) +\gamma_*(S(\al(a_{n})),
$$
and hence by \eqref{formula szpiegula666},  we get
\begin{align*}
\widetilde{S}(\phi_n(q(a_{0})\oplus ... \oplus a_{n}))&=\widetilde{S}(\phi_{n+1}(q(a_{0})\oplus  ... \oplus q(a_{n}) \oplus \al(a_n)))
\\
&=\widetilde{S}(\phi_{n+1}(T_n(q(a_{0})\oplus ... \oplus a_{n}))). 
\end{align*}
Hence $\widetilde{S}:B\to C$ is an injective homomorphism. 
\\
Clearly,  $S= \widetilde{S}\circ T$. In particular, $\S:B\to C $ is  non-degenerate because $S:A\to C$ is.
 Furthermore,  for $a\in \overline{\gamma}^k(1)C\overline{\gamma}^k(1)$, $k>0$, we have
 $$
 \gamma(\gamma_*^k(a)))=\overline{\gamma}(1) \gamma_*^{k-1}(a) \overline{\gamma}(1)=\gamma_*^{k-1}\left(\overline{\gamma}^{k}(1)a\overline{\gamma}^k(1)\right)=\gamma_*^{k-1}\left(a\right).
 $$
Therefore, in view of \eqref{formula szpiegula666}, since  $\gamma(Q)=0$, we get 
\begin{align*}
\gamma(\S(\phi_n(	q(a_{0})\oplus q(a_1) \oplus ... \oplus a_{n})))
 &=  S(Qa_{1}) +  \gamma_*(QS(a_{2}))+... + \gamma_*^{n-1}(S(a_{n}))\\
 &=\S(\beta(\phi_n(a_{0}\oplus a_1 \oplus ... \oplus a_{n}))).
\end{align*}
Hence  $\gamma \circ \S=\S\circ \beta$ and   $(B,\beta)  \stackrel{\widetilde{S}}{\rightarrow}   (C,\gamma)$ is an embedding. 
Moreover, since
$$
\overline{\S}(\overline{\beta}_*(1))=\overline{\S}\left(\overline{\phi}_1(0 \oplus \overline{\alpha}(1))\right)=  \overline{\gamma}_*( \overline{S}(\overline{\alpha}(1)))=\overline{\gamma}_*( \overline{\gamma}(1))= \overline{\gamma}_*(1),
$$ 
$\S$ is  covariant    by Lemma \ref{lemma for reversibles}.
\\
This proves the existence part. Uniqueness of $\S$ follows immediately from the relations $B=\sum_{n=0}^\infty \beta_*^n(T(A))$ and $\gamma_* \circ \S=\S\circ \beta_*$, cf. Lemma \ref{lemma for reversibles}.

iii). Suppose  $(A,\alpha) \stackrel{S}{\rightarrow} (C,\gamma)$ is an embedding of $(A,\alpha)$ into a reversible $C^*$-dynamical system $(C,\gamma)$. If
$C=\sum_{n=0}^\infty \gamma_*^n(S(A))$  and $J=\{a\in A: \overline{\gamma}_*(1)S(a)=S(a)\}$, then $S$ is a $J$-embedding, see  Lemma \ref{lemma describing ideal J}, and in view of item i) the embedding $(B,\beta) \stackrel{\S}{\rightarrow} (C,\gamma)$ whose existence is guaranteed by item ii) is surjective. Hence $(B,\beta)$ and $(C,\gamma)$ are equivalent.
\\
If $(C,\gamma)$ possess the property described in item ii) then there is a covariant embedding $(C,\gamma) \stackrel{\widetilde{T}}{\rightarrow} (B,\beta)$ such that $T= \widetilde{T}\circ S$. Appealing again to Lemma \ref{lemma for reversibles} we get  
$$
\T\circ \gamma_*\circ S=\beta_*\circ \T\circ S  =\beta_*\circ T.
$$
Thus by item i) the embedding $\widetilde{T}$ is onto $B$. Consequently, $(B,\beta)$ and $(C,\gamma)$ are equivalent.  
 \end{proof}
We extend \cite[Definition 3.4]{kwa-ext} to  non-unital case as follows.
\begin{defn}
Let  $(A,\alpha)$ be a $C^*$-dynamical system,  $J$  an ideal in $(\ker\alpha)^\bot$ and $(B,\beta)$ the direct limit of \eqref{ciag prosty C*-algebr}.  We call $(B,\beta)$ the  \emph{natural reversible $J$-extension} of $(A,\alpha)$. If $J=(\ker\alpha)^\bot$ we drop the prefix `$J$-'.
\end{defn} 
\begin{rem}\label{identifications remark} In the sequel we will usually identify $A$ with the corresponding subalgebra of $B$. Then the natural reversible $J$-extension $(B,\beta)$ can be characterized as  a reversible $C^*$-dynamical system satisfying  
$$
B=\sum_{n=0}^\infty \beta_*^n(A), \qquad \overline{\beta_*(1)}A\cap A=J, \qquad \beta|_{A}=\alpha.
$$
To extend the above picture, so that it covers also the algebras $B_n$, $n\in \N$, we put  $q:=1-\overline{\beta_*(1)}$ (this  makes a perfect match with our previous notation on the quotient map $A\to A/J$). We   identify the algebra $B_n$, $n\in \N$, with its image in $B$ by the mapping
$$
B_n \ni \bigoplus_{k=0}^n q(a_k)\oplus a_n \longmapsto \sum_{k=0}^{n-1}\beta_*^k(q a_k) +\beta_*^n(a_n) \in B.
$$
With this identification we have 
\begin{equation}\label{direct sum decomposition}
B_n=\sum_{k=0}^{n-1}\beta_*^k(qA_k) +\beta_*^n(A_n)=\sum_{k=0}^{n-1}\beta_*^k(q) \beta_*^k(A) +\beta_*^n(A)=\sum_{k=0}^{n-1}\beta_*^k(q) B +\beta_*^n(A).%
\end{equation} 
The above sums are actually direct sums of $C^*$-algebras because the projections $\overline{\beta}_*^n(1)$ and $\overline{\beta}_*^k(q)=\overline{\beta}_*^k(1)-\overline{\beta}_*^{k+1}(1)$, $k=0,...,n-1$, are   pairwise orthogonal and central in $M(B)$,   cf. Proposition \ref{lemma for extensions of transfers}.
\end{rem}

\subsection{Invariant ideals and $J$-pairs}\label{Invariant ideals and J-pairs subsection}
Now we introduce invariant ideals and discuss the relevant notions and facts that  will lead us to description of such ideals in natural reversible extensions defined in the previous subsection.
   \begin{defn}
   We  say that an ideal 
  $I$  in $A$ is a \emph{positively invariant} ideal in $(A,\alpha)$ if $\al(I)\subseteq I$. Any such ideal $I$ induces a quotient $C^*$-dynamical system $(A/I,\al_I)$ where $\al_I(a+I)=\al(a)+I$ for all $a\in A$. We will refer to $(A/I,\al_I)$ as a \emph{subsystem} of $(A,\al)$.
   \end{defn}
 When  regarding negative invariance, as in the case  of extensions, it is useful to consider a  family of  such  notions  parametrized by an ideal. 
\begin{defn}
Let $I$ and $J$ be ideals in $A$ where $J\subseteq (\ker\al)^\bot$. We  say that
  $I$ is \emph{$J$-negatively invariant} ideal in $(A,\alpha)$ if $J\cap \al^{-1}(I)\subseteq I$.
If $I$ is both positively invariant and $J$-negatively invariant we   say that $I$ is \emph{$J$-invariant}. If $J=(\ker\alpha)^\bot$ we drop the prefix `$J$-'.
\end{defn}
When $\alpha$ is an automorphism invariance of  $I$ means that $\alpha(I)=I$.  More generally, we have
\begin{lem}\label{helping lemma}
Suppose $(A,\al)$ is a reversible $C^*$-dynamical system and let $I$ be an ideal in $A$. The following conditions are equivalent:
\begin{itemize}
\item[i)] $I$ is invariant in  $(A,\al)$, 
\item[ii)] $\al(I)\subseteq I$ and $\al_*(I)\subseteq I$,
\item[iii)]  $\al(I)=\overline{\al}(1)I\overline{\al}(1)$,
\item[iv)] $\al_*(I)=\overline{\al}_*(1)I$.
\end{itemize}
In particular, if $I$ is invariant  the system $(A/I,\al_I)$  is  reversible and the set of invariant ideals in $(A,\alpha)$, ordered by inclusion, is a lattice.
\end{lem}
\begin{proof}
i)$\Leftrightarrow$ ii). By the form of the complete transfer operator, see \eqref{complete transfer operator form}, we have $\alpha^{-1}(I)=\ker\alpha\oplus \alpha_*(I)$. Thus negative invariance of $I$ is equivalent to the inclusion $\al_*(I)\subseteq I$.

ii)$\Rightarrow$ iii). Inclusion $\al(I)\subseteq I$ implies that $\al(I)\subseteq\overline{\al}(1)I\overline{\al}(1)$, while both of the inclusions in ii)  imply that $\overline{\al}(1)I\overline{\al}(1)=\alpha(\alpha_*(I))\subseteq I$. 

iii)$\Rightarrow$ iv). We have $\al_*(I)=\al_*(\overline{\al}(1)I\overline{\al}(1))=\al_*(\al(I))=\overline{\al}_*(1)I$.

iv)$\Rightarrow$ ii). We have $\al(I)=\al(\overline{\al}_*(1)I)=\alpha(\alpha_*(I))=\overline{\al}(1)I\overline{\al}(1)\subseteq I$ and  $\al_*(I)=\overline{\al}_*(1)I\subseteq I$.
\\
The remaining part of the assertion is now straightforward.
\end{proof}

Note that an ideal $I$ is $J$-invariant if and only if the system $(A/I,\al_I)$ is well defined and $q_I(J)\subseteq (\ker\al_I)^\bot$.
 Thus  $J$-invariant ideals correspond to    subsystems of $(A,\al)$ which `admit $J$-extensions'. 
Alternatively,   $J$-invariant ideals in $(A,\alpha)$ can be viewed as parts of  invariant ideals in $J$-extensions of $(A,\alpha)$.
\begin{lem}\label{proposition ktore winno byc lemma}
Let  $T$ be  a $J$-embedding of $(A,\alpha)$ into a certain $C^*$-dynamical system $(B,\beta)$. If $\I$ is an invariant ideal in in $(B,\beta)$ then 
\begin{equation}\label{obciecie idealu}
I:=T^{-1}(\I)
\end{equation}
is  a $J$-invariant ideal in $(A,\alpha)$. In particular, $T$ factors through to an $q_I(I')$-embedding
$(A/I,\alpha_I) \stackrel{T_{\I}}{\longrightarrow} (B/\I,\beta_{\I})$  where $I'$ 
is an  ideal in $A$ such that 
\begin{equation}\label{J-pair relations}
J\subseteq I' \quad \textrm{and}\quad  I'\cap \al^{-1}(I)=I.
\end{equation}
If $(B,\beta)$ is reversible then $I'=\{a\in A: (1-\overline{\beta}_*(1))T(a)\in \I\}$.
\end{lem}
\begin{proof} Since $
T(\alpha(I))=\beta(T(I))=\beta(T(T^{-1}(\I)))\subseteq \beta(\I)\subseteq \I
$ we have  $\al(I)\subseteq I$. To show  the negative $J$-invariance of $I$   note first that
$$
T(\alpha^{-1}(I)) = T( (T\circ \alpha)^{-1}(\I))=T( (\beta\circ T)^{-1}(\I))=T (T^{-1} (\beta^{-1}(\I)))\subseteq \beta^{-1}(\I).
$$
Therefore by negative invariance of $\I$ we have
$$
T(J\cap \alpha^{-1}(I)) \subseteq (\ker\beta)^\bot\cap T(\alpha^{-1}(I))  \subseteq (\ker\beta)^\bot\cap \beta^{-1}(\I)\subseteq \I.
$$
Hence $J\cap \al^{-1}(I)\subseteq I$, and $I$ is $J$-invariant.   Plainly, $T_{\I}(a+I)=T(a)+\I$ defines an embedding of the subsystem $(A/I,\alpha_I)$ into  $(B/\I,\beta_{\I})$. By definition $T_{\I}$ is $T_{\I}^{-1}((\ker\beta_{\I})^\bot)$-covariant. Letting  $I':= q_I^{-1}(T_{\I}^{-1}((\ker\beta_{\I})^\bot))$ we obtain the assertion. In particular, since $\alpha^{-1}(I)=q_I^{-1}(\ker\alpha_I)=q_I^{-1}(T^{-1}_{\I}(\ker\beta_{\I}))$, the relations            
$$
J\cap \al^{-1}(I)\subseteq I, \qquad T_{\I}^{-1}((\ker\beta_{\I})^\bot) \cap \ker\alpha_I =\{0\} 
$$    
imply  \eqref{J-pair relations}.  
\\
Suppose now that $(B,\beta)$ is reversible and  note that by  Lemma \ref{helping lemma} also the subsystem $(B/\I,\beta_{\I})$ is reversible; its  complete transfer operator is given by $\beta_{*,\I}(b +\I)= \beta_*(b)+\I$, $b\in B$. Hence applying  Lemma \ref{lemma describing ideal J} to the embedding $T_{\I}$ we get 
\begin{align*}
a \in I' &\,\, \Longleftrightarrow \,\, q_{I}(a)\in T_{\I}^{-1}((\ker\beta_{\I})^\bot) \,\, \Longleftrightarrow \,\, \overline{\beta}_{*,\I}(1)T_{\I}(q_{I}(a))=T_{\I}(q_{I}(a)) 
\\
&\,\, \Longleftrightarrow \,\,  \overline{\beta}_{*}(1)T(a) +\I=T(a)  + \I \,\, \Longleftrightarrow \,\, (1-\overline{\beta}_{*}(1))T(a) \in \I.
\end{align*}  
\end{proof}
\begin{rem}\label{remark on J-pairs}
In general the ideal $I'$ above carries an essential data concerning the ideal $\I$, which is not visible from the standpoint of   $I$ and $J$. Indeed,  when $I$ is a  $J$-invariant ideal then $I+J$ is the minimal ideal satisfying \eqref{J-pair relations}. On one hand, in the favorable situation when $A=\ker\alpha \oplus (\ker\alpha)^\bot$ and $J=(\ker\alpha)^\bot$,  any ideal satisfying \eqref{J-pair relations} has to be of the form  $I'=I+J$  (as we then have $I'=(I'\cap \ker\alpha)\oplus(I'\cap (\ker\alpha)^\bot)\subseteq I + J$). On the other hand, as the following example shows when $\ker\alpha$ is not  complemented  in $A$   there might be ideals $I'$ satisfying  \eqref{J-pair relations} that are  strictly larger than $I+J$.
 \end{rem}
\begin{ex}\label{example on J-pairs}
Let $A=C([0,1])$ and $\alpha(a)=a\circ \varphi$ where $\varphi:[0,1]\to [0,1]$ is given by  $\varphi(x):=x/2$. Put $J=(\ker\alpha)^{\bot}=C_{0}([0,\frac{1}{2}))$ and  let $I$ be the ideal $C_0([0,1]\setminus K)$ where $K=\{0,1,\frac{1}{2},\frac{1}{4},\frac{1}{8},...\}$. Since  $\alpha(I)=C_0\left([0,1]\setminus \varphi^{-1}(K)\right)=I$ and $\alpha^{-1}(I)=C_0([0,1]\setminus \varphi(K))\subseteq I$,   the ideal $I$ is positively invariant and both of the ideals
$$
I'_1:=C_0([0,1)), \qquad I'_2:=C_0([0,1)\setminus\{1/2\})
$$
satisfy  relations \eqref{J-pair relations}. Actually these are the only such ideals. The only difference between $I'_1$ and $I'_2$ is `located' in  the boundary $\{1/2\}$ of the hull of $\ker\alpha$ (which is not empty because $\ker\alpha$ is not complemented in $A$).
\end{ex}
The above remark and example lead us to  the following definition.
\begin{defn}
Let $I, I', J$ be ideals in $A$ where $J\subseteq (\ker\al)^\bot$. We  say that $(I,I')$ is a \emph{$J$-pair} for $(A,\al)$  if 
$$ 
I  \textrm{ is positively invariant,}\quad  J\subseteq I' \quad \textrm{and}\quad  I'\cap \al^{-1}(I)=I.
$$
Note that then $I$ is automatically $J$-invariant. We equip  the set of $J$-pairs for $(A,\al)$ with a natural  partial order induced by inclusion: $(I_1, I_1') \subseteq (I_2,I_2')$ $\stackrel{def}\Longleftrightarrow$  $I_1\subseteq I_2$ and $I_1'\subseteq I_2'$.
\end{defn}
 $J$-pairs arise naturally from morphisms.
\begin{lem}\label{explaining J-pairs}
A morphism  $(A,\alpha)\stackrel{T}{\rightarrow} (B,\beta)$ is  covariant on $J\subseteq (\ker\alpha)^\bot$  if and only if  
$$
I:=\ker T \quad\textrm{and}\quad I':= T^{-1}((\ker\beta)^\bot)
$$
form a $J$-pair $(I,I')$.
\end{lem}
\begin{proof} Since $I=T^{-1}(\{0\})$ is positively invariant and $I\subseteq I'$ we have $I\subseteq I'\cap \al^{-1}(I)$. 
The reverse inclusion holds because for $a\in   I'\cap \al^{-1}(I)$ we have 
$T(a)\in  (\ker\beta)^\bot$ and $\beta(T(a))=T(\alpha(a))=0$, which implies that  $a\in \ker T=I$.
Hence $I'\cap \al^{-1}(I)=I$. Thus $(I,I')$ is a $J$-pair if and only if  $T$ is  covariant on $J$.

\end{proof}

\subsection{Invariant ideals in natural reversible extensions}\label{Invariant ideals in natural reversible extensions}
Now we are in a position to provide the description of invariant ideals in  natural reversible  extensions.
\begin{thm}\label{invariant ideals and J-pairs thm}
Let $(B,\beta)$ be a natural reversible $J$-extension of $(A,\al)$ and retain the identifications and notation from 
Remark \ref{identifications remark}. We have an order preserving bijection between $J$-pairs $(I,I')$ for $(A,\al)$  and invariant ideals $\I$ in $(B,\beta)$,  given by 
\begin{equation}\label{J-pair from invariant ideal}
I=\widetilde{I} \cap A, \qquad I'=\{a\in A: q a \in \widetilde{I}\}, 
\end{equation}
\begin{equation}\label{invariant ideal from J-pair}
 \I \cap B_n= \sum_{k=0}^{n-1} \beta_*^k(qI') + \beta_*^n(I), \qquad n\in \N.
\end{equation}
In particular, the set of $J$-pairs is a lattice. Moreover, denoting by $(B_{(I,I')},\beta_{(I,I')})$ the natural reversible $q_I(I')$-extension of the subsystem $(A_{I},\alpha_I)$ we have 
$$
(B/\I,\beta_{\I})\cong (B_{(I,I')},\beta_{(I,I')})
$$ 
where $\I$ is the invariant ideal corresponding to a $J$-pair $(I,I')$.
\end{thm}
\begin{proof} 
Let $\I$ be an invariant ideal in $(B,\beta)$.  Lemma  \ref{proposition ktore winno byc lemma} tells us  that \eqref{J-pair from invariant ideal} defines a $J$-pair $(I,I')$. To show that $\I$ is determined by $(I,I')$  it suffices to verify \eqref{invariant ideal from J-pair}. 
But the latter is simple. Indeed, taking into account the direct sum decomposition \eqref{direct sum decomposition}, we get that  the sum $\sum_{k=0}^{n-1} \beta_*^k(qa_k) + \beta_*^n(a_n)$, where $a_k\in A_k$, is in $\I$ if and only if  its summands are in $\I$.  By Lemma \ref{helping lemma}  we have $\beta(\I)\subseteq \I$ and $\beta_*(\I)\subseteq \I$. Hence we see  that  $\beta_*^k(qa_k)$ is in $\I$ if and only if $a_k \in I'$.  Similarly $\beta_*^n(a_n)$ is in $\I$ if and only if $a_n \in I$.   

Accordingly, we have shown that \eqref{J-pair from invariant ideal} yields an injective map from the set of invariant ideals in $(B,\beta)$ to the set of $J$-pairs for $(A,\alpha)$. To prove it is surjective let $(I,I')$ be an arbitrary $J$-pair for $(A,\alpha)$ and view $(B,\beta)$ as the direct limit of \eqref{ciag prosty C*-algebr}. Put $I_k:=\overline{\alpha}^{k}(1)I\overline{\alpha}^{k}(1)$ and $I_k':=\overline{\alpha}^{k}(1)I'\overline{\alpha}^{k}(1)$, $k\in \N$. For each  $n\in \N$,  the formula
$$ \I_n:=q(I_0')\oplus q(I_{1}')\oplus ... \oplus q(I'_{n-1}) \oplus I_n 
$$ 
defines an ideal in $B_n$ (recall that $q:A\to A/J$ is the quotient map). Relations saying that $(I,I')$ is a $J$-pair  imply that $\I_{n+1}\cap T_n(B_n)=T_n(\I_n)$,  $n\in \N$. Hence with the identifications from Remark \ref{identifications remark}, we deduce that the closure $\I$ of the ascending sum $\bigcup_{n\in \N} \I_n$  is an ideal in $B$ satisfying \eqref{J-pair from invariant ideal} and \eqref{invariant ideal from J-pair}. We still need to show that $\I$ is invariant in $(B,\beta)$. Using 
description of $\beta_*$ given in the proof of Theorem \ref{granica prosta zepsute twierdzenie} we get 
$$
\beta_*(\I_n)=0\oplus q(I_1')\oplus q(I_{2}')\oplus ... \oplus q(I'_{n}) \oplus I_{n+1} \subseteq \I_{n+1}.
$$
Therefore $\beta_*(\I)\subseteq \I$. Similarly, we have $\beta(\I)\subseteq \I$ because 
  $$
\beta(\I_{n+1})= q(I_1')\oplus q(I_{2}')\oplus ... \oplus q(I'_{n}) \oplus I_{n+1} \subseteq \I_{n}.
$$
Consequently,  $\I$ is invariant by Lemma \ref{helping lemma}.
 This shows the  bijective correspondence between $J$-pairs in $(A,\alpha)$ and invariant ideals in $(B,\beta)$. Plainly, it  preserves the order. In particular, by Lemma \ref{helping lemma} both the set of invariant ideals in $(B,\beta)$ and the set of $J$-pairs in $(A,\alpha)$ are lattices.

Now  fix an  ideal $\I$ invariant in $(B,\beta)$ and let $(I,I')$ be the corresponding $J$-pair. 
Let $(B^{\I},\beta^{\I})$  be the direct limit  of the direct sequence 
$$
(B_0^{\I},\beta_0^{\I}) \stackrel{T_0^{\I}}{\longrightarrow} (B_1^{\I},\beta_1^{\I})
\stackrel{T_1^{\I}}{\longrightarrow} (B_2^{\I},\beta_2^{\I})  \stackrel{T_2^{\I}}{\longrightarrow}...\,
$$  where 
$$%\begin{equation}\label{abstract algebras B_n}
B_n^{\I}:=q_{I'}(A_{0})\oplus q_{I'}(A_{1})\oplus ... \oplus q_{I'}(A_{n-1}) \oplus q_{I}(A_{n}), 
$$%\end{equation}
$$%\begin{equation}\label{beta_n formula}
\beta_n^{\I}(q_{I'}(a_{0})\oplus  ... \oplus q_{I'}(a_{n-1})\oplus q_{I}(a_{n})):=q_{I'}(a_{1})\oplus ... \oplus q_{I'}(a_{n})\oplus q_I(\al(a_n))
$$%\end{equation}
and
$$
T_n^{\I}(q_{I'}(a_{0})\oplus  ... \oplus q_{I'}(a_{n-1})\oplus q_{I}(a_{n}))= q_{I'}(a_{0})\oplus ... \oplus q_{I'}(a_{n})\oplus q_I(\al(a_n))
$$
for $a_k \in A_k$, $k=0,...,n$.
Note that, as $\alpha(I)\subseteq I$ and $I\subseteq I'$,  the  $C^*$-dynamical systems $(B_n^{\I},\beta^{\I}_n)$ and their bonding maps $T_n^{\I}$ are well defined. 

Since $I, J\subseteq I'$ we have the natural isomorphisms 
\begin{equation}\label{natural isomorphisms}
(A/J)/q(I') \cong  A/I' \cong (A/I)/q_I(I').
\end{equation}
It is not hard to convince yourself  that they induce natural equivalences
$$
 (B/\I,\beta_{\I})\cong (B^{\I},\beta^{\I})\cong (B_{(I,I')},\beta_{(I,I')}).
$$
Indeed, let $
(B_0^{(I,I')},\beta_0^{(I,I')}) \stackrel{T_0^{(I,I')}}{\longrightarrow} (B_1^{(I,I')},\beta_1^{(I,I')})
\stackrel{T_1^{(I,I')}}{\longrightarrow} (B_2^{(I,I')},\beta_2^{(I,I')})  \stackrel{T_2^{(I,I')}}{\longrightarrow}...\,
$ be the direct sequence defining  $(B_{(I,I')},\beta_{(I,I')})$ and retain the notation form the first part of the proof. 
Then the isomorphisms \eqref{natural isomorphisms} extend (by direct sums and restrictions) to isomorphisms making the following diagram commutative
$$
\begin{xy}
\xymatrix@C=3pt{
    B_n^{\I} \ar[d]^{T_n^{\I}}  &  \cong  &    B_n/\I_n  \cong q_{\I}(B_n)\cong B_n/\I_n \ar[d]^{T_n}&  \cong  &  B_n^{(I,I')} \ar[d]^{T_n^{(I,I')}}   \\
    B_{n+1}^{\I}   &  \cong  &     B_{n+1}/\I_{n+1}  \cong q_{\I}(B_{n+1})\cong B_{n+1}/\I_{n+1} & = &  B_n^{(I,I')}
        }
  \end{xy}.
$$
The horizontal isomorphisms establish equivalences $(B_n^{\I},\beta^{\I}_n)\cong (B_n/\I_n, (\beta_n)_{\I_n})\cong (B_1^{(I,I')},\beta_1^{(I,I')})$ which imply equivalences of the limiting systems. 
\end{proof}
\begin{cor}\label{complemented kernel corollary}
Suppose $\ker\al$ is a complemented ideal in $A$ and let $(B,\beta)$ be a natural reversible extension of $(A,\al)$. We have an order preserving bijection between  the invariant ideals  in systems $(A,\al)$  and $(B,\beta)$.  It is given by the relations 
\begin{equation}\label{bijection between A -generated and J-invariant}
I=\I \cap A, \qquad \quad \I=\sum_{n=0}^\infty \beta_*^n(I).
\end{equation}
Moreover, the subsystem $(B/\I,\beta_{\I})$  corresponding to $\I$ is  equivalent to the natural reversible extension of the subsystem $(A/I,\alpha_I)$ corresponding to $I$. 
\end{cor}
\begin{proof} Within the notation of Theorem \ref{invariant ideals and J-pairs thm} we necessarily  have  $I'=I+(\ker\alpha)^\bot$ (see Remark \ref{remark on J-pairs}) and $q=1-\overline{\beta}_*(1)$ projects $A$ onto $\ker\alpha$. Therefore $qI'=I\cap \ker\alpha\subseteq I$, and consequently formula \eqref{invariant ideal from J-pair}
 implies that $\I=\sum_{n=0}^\infty \beta_*^n(I)$. Since $q_I(I')=q_I((\ker\alpha)^\bot)=(\ker\alpha_I)^\bot$ the second part of the assertion follows from Theorem \ref{invariant ideals and J-pairs thm}.
\end{proof}
\begin{rem}\label{remark on A-generated ideal in B}
For an arbitrary $C^*$-dynamical system $(A,\alpha)$ and an ideal $J$ in $(\ker\alpha)^\bot$ relations \eqref{bijection between A -generated and J-invariant} establish an order preserving bijection between the $J$-invariant ideals in $(A,\alpha)$ and invariant ideals in $(B,\beta)$ that are generated by their intersection with $A$. Indeed, the latter ideals correspond via bijection from Theorem \ref{invariant ideals and J-pairs thm}  to $J$-pairs of the form $(I,I+J)$ and as $q(I+J)=qI$ the formula \eqref{invariant ideal from J-pair}  reduces to  $\I=\sum_{n=0}^\infty \beta_*^n(I)$.
\end{rem}
It is not an immediate fact  that the set of $J$-invariant ideals in $(A,\alpha)$ is a lattice.
\begin{cor}\label{corollary}
For any $C^*$-dynamical system $(A,\alpha)$ and any ideal $J$ in $(\ker\alpha)^\bot$ the set of $J$-invariant ideals in $(A,\alpha)$ is a lattice.
\end{cor}
\begin{proof}
Let $I_1$ and $I_2$ be $J$-invariant ideals. It is evident that the intersection $I_1\cap I_2$  is a $J$-invariant ideal, which yields a natural meet operation.  For the existence of join operation note that  $(I_1,I_1+J)$ and $(I_2,I_2+J)$ are $J$-pairs,   and therefore by Theorem \ref{invariant ideals and J-pairs thm}  there exists the join $J$-pair:
$$
(I,I')=(I_1,I_1+J) \vee (I_2,I_2+J).
$$
Plainly, $I$ is a $J$-invariant ideal containing $I_1$ and $I_2$. Moreover, if  $I_1,I_2 \subseteq K$ for a certain  $J$-invariant ideal $K$ then
$$(I_1,I_1+J), (I_2,I_2+J) \subseteq (K, K+J)\, \Longrightarrow \, (I,I')\subseteq (K, K+J) \, \Longrightarrow  I\subseteq K.$$
That is, $I$ is the minimal $J$-invariant ideal containing $I_1, I_2$. Hence $J$-invariant ideals form a lattice.
\end{proof}
Another consequence of Theorem \ref{invariant ideals and J-pairs thm} is that the universal property of natural reversible $J$-extensions described in item ii) of Theorem \ref{granica prosta zepsute twierdzenie} can be extended to not necessarily injective morphisms.
\begin{cor}\label{corollary for universality}
Let $(B,\beta)$ be a natural reversible  $J$-extension of $(A,\al)$. If    $(A,\alpha)\stackrel{S}{\rightarrow}(C,\gamma)$ is a morphism covariant on  $J$ and $(C,\gamma)$ is reversible then $S$ extends uniquely to a covariant morphism $(B,\beta)\stackrel{\S}{\rightarrow}(C,\gamma)$.
\end{cor}
\begin{proof}
By Lemma \ref{explaining J-pairs},  $I:=\ker S$ and $I':= S^{-1}((\ker\gamma)^\bot)$ form a $J$-pair for $(A,\alpha)$. Clearly,   $S$ factors through to a $q(I')$-embedding $(A/I,\alpha_I)\stackrel{S_I}{\rightarrow}(C,\gamma)$. Hence by  Theorem \ref{granica prosta zepsute twierdzenie}, using the  notation and equivalence from Theorem \ref{invariant ideals and J-pairs thm}, $S_I$ extends uniquely to the covariant embedding  $(B/\I,\beta_{\I})\cong (B_{(I,I')},\beta_{(I,I')})  \stackrel{\S_I}{\rightarrow} (C,\gamma)$.  
Thus $\S=\S_I\circ q_{\I}$ defines the desired extended morphism. It is unique because the  invariant ideal $\I=\ker\S$ is   uniquely determined by $(I,I')$.
\end{proof}

\section{Representations of $C^*$-dynamical systems. Crossed products}\label{crossed products section}
We adopt the  point of view that crossed products are $C^*$-algebras encoding representation theory of $C^*$-dynamical systems. A fundamental remark for our analysis  is that any representation of $(A, \alpha)$ can be viewed as a morphism to a reversible $C^*$-dynamical system. Thus the notion of covariance  transfers  naturally from morphisms to representations. This leads us to the introduction of relative crossed products $C^*(A,\alpha,J)$.
\subsection{Covariant representations and morphisms to reversible systems}\label{covariant subsection}

Representations of $(A,\alpha)$ are defined in the following straightforward manner, cf. \cite[Definition 1.1]{kwa-leb}.
\begin{defn}\label{kowariant  rep defn*}
A \emph{Representation} of  a $C^*$-dynamical system $(A,\alpha)$ on a Hilbert space $H$  is a pair $(\pi,U)$
where $\pi:A\to \B(H)$ is a non-degenerate representation and $U\in \B(H)$ is such that 
\begin{equation}\label{covariance rel1*}
U\pi(a)U^* =\pi(\alpha(a)),\qquad a \in A.
\end{equation}
If $\pi$ is injective we say $(\pi,U)$ is a \emph{faithful}. We denote by $C^*(\pi,U)$ the $C^*$-algebra  generated by $\pi(A)\cup\pi(A)U$ and refer to it as to the $C^*$-algebra generated by $(\pi,U)$.
 \end{defn}
\begin{rem}\label{remark for one proof and for Bialorusy}
One can show (see Proposition \ref{on existence of crossed products} below) that  if  $(\pi,U)$ is a representation of $(A,\alpha)$ then $\pi:A \to C^*(\pi,U)$ is a non-degenerate homomorphism and  $U\in M(C^*(\pi,U))$. Thus  we could consider abstract representations $(\pi,U)$ of $(A,\alpha)$ where 
 $\pi:A\to C$ is a non-degenerate homomorphism into a $C^*$-algebra $C$ and $U\in M(C)$ is such that \eqref{covariance rel1*} holds. But then composing $\pi$ with any non-degenerate representation of $C$ on  $H$  we get back to Definition \ref{kowariant  rep defn*}.
 \end{rem}
It is a remarkable consequence of multiplicativity of $\alpha$ that  any representation of $(A,\alpha)$ defines a morphism to a reversible $C^*$-dynamical system.
 \begin{prop}\label{iteration of representations}
Let  $(\pi,U)$ be  a representation of   $(A,\alpha)$ in a $C^*$-algebra $C$. Then $U$ is a (power) partial isometry, 
$$
B:=\sum_{n=0}^\infty U^{*n}\pi(A)U^n
$$
is a $C^*$-algebra and 
\begin{equation}\label{relations for concrete reversibles}
U BU^*\subseteq B, \qquad U^*BU\subseteq B, \qquad U^*U\in B'
\end{equation}
where $B'$ is the commutant of $B$.
In particular,  putting $\beta(\cdot):=U(\cdot) U^*$ we get that $(B,\beta)$  is a reversible $C^*$-dynamical system with the complete transfer operator given by $\beta_*(\cdot):=U^*(\cdot) U$.
\end{prop}

\begin{proof}
Since $\pi:A\to \B(H)$ is non-degenerate it extends uniquely to the strictly continuous unital homomorphism $\overline{\pi}:M(A)\to \B(H)$. Moreover, non-degeneracy of $\pi$ and strict continuity of $\overline{\alpha}$ readily imply that $(\overline{\pi},U)$ is a representation of  $(M(A),\overline{\alpha})$. Since for each $n\in \N$,  $\pi(\overline{\alpha}^n(1))=U^{*n}U^n$ is a projection, $U$ is a power partial isometry. By  \cite[Lemma 1.2]{kwa-leb}   we have $U^{*n}U^n \in \overline{\pi}(M(A))'$ and hence all the more  $U^{*n}U^n \in \pi(A)'$.
\\
Let $\beta(\cdot):=U(\cdot) U^*$ and $\beta_*(\cdot):=U^*(\cdot) U$.
For each $N\in \N$, the self-adjoint linear space $\sum_{n=0}^N \beta_*^k(\pi(A))$ is a $*$-algebra because for  $l\leq k$ and $a,b\in A$, we have 
\begin{align*}
\beta_*^k(\pi(a)) \beta_*^l(\pi(b))
&=U^{*k} \pi(a) \beta^k(1)U^{k-l}  \beta^l(1)\pi(b)U^l
\\
&=U^{*k} \pi(a \overline{\alpha}^k(1))U^{k-l} (U^{*(k-l)}U^{k-l})\pi(\overline{\alpha}^k(1)b)U^l
\\
&=U^{*k} \pi(a \overline{\alpha}^k(1))U^{k-l} \pi(\overline{\alpha}^l(1)b) U^{*(k-l)}U^{k}
\\
&=\beta_{*}^k\left(\pi\big(a \overline{\alpha}^k(1)  \alpha^{k-l}(b)\big)\right) \in \beta_*^k(\pi(A)).
\end{align*}
Hence $B:=\sum_{k=0}^\infty \beta_*^k(\pi(A))$ is a $C^*$-algebra. Relations $U^*BU\subseteq B$ and $U^*U\in B'$ are now straightforward. Moreover, for any $a\in A$ and $k>0$ the projection $U^{*k-1}U^{k-1}$ commute with  $\overline{\pi}(\overline{\alpha}(1))\in \overline{\pi}(M(A))$  and thus we have
\begin{align*}
\beta(\beta_*^k(\pi(a)))&=\overline{\pi}(\overline{\alpha}(1))  (U^{*k-1}U^{k-1})  U^{*k-1}\pi(a) U^{*k-1}  (U^{*k-1}U^{k-1})
 \overline{\pi}(\overline{\alpha}(1))
 \\
&  =\beta_*^{k-1}(\pi(\overline{\alpha}^k(1) a\overline{\alpha}^k(1))).
\end{align*}
This implies that $U BU^*\subseteq B$.
Using relations \eqref{relations for concrete reversibles} it is easy to see that $(B,\beta)$ is a reversible $C^*$-dynamical system. Indeed, for $a,b \in B$ we have 
$$
\beta(ab) =U (U^*U)a bU^*=Ua (U^*U) bU^*=\beta(a)\beta(b),
$$
 $$
 a\beta_*(b)=a (U^*U)U^*b U=(U^*U)a U^*b U=\beta_*(\beta(a)b)
 $$
 and $\beta(\beta_*(a))=\beta(1)a\beta(1)$.

  \end{proof}
\begin{rem}\label{from morphisms to representations} Let $(\pi, U)$ be a representation of $(A,\alpha)$ and let $(B,\beta)$ be the associated $C^*$-dynamical system defined in the assertion of  Proposition \ref{iteration of representations}. Treating $\pi$ as a morphism   $(A,\alpha) \stackrel{\pi}{\rightarrow}(B,\beta)$ we have 
$$
\{ a\in A: U^*U \pi(a)=\pi(a)\}=\pi^{-1}((\ker\beta)^\bot),
$$
cf. Lemma \ref{lemma describing ideal J}.
Thus if $J$ is an ideal in $(\ker\alpha)^\bot$ then the morphism $(A,\alpha) \stackrel{\pi}{\rightarrow}(B,\beta)$ is covariant on $J$  if and only if the representation $(\pi,U)$ is covariant on $J$ in the following sense. 
\end{rem}
\begin{defn}[cf.  \cite{kwa-leb}, Definition 1.7]\label{kowariant  rep defn 2} 
Let $(\pi,U)$ be a representation of  $(A,\alpha)$ and let $J$ be an ideal in $(\ker\alpha)^\bot$. We say that $(\pi,U)$ is \emph{covariant on $J$} if 
\begin{equation}\label{covariance rel3}
J\subseteq \{ a\in A: U^*U \pi(a)=\pi(a)\}.
\end{equation}
If  the above inclusion is the equality we say   $(\pi,U)$ is  a \emph{$J$-covariant representation}. 
If  
 $(\pi,U)$  is covariant on $(\ker\alpha)^\bot$  we say $(\pi,U)$ is  a \emph{covariant representation}. 
\end{defn}
By  Remark \ref{from morphisms to representations} we   can immediately translate Lemmas \ref{covariance for extensions}, \ref{covariance in terms of projection} and  Proposition  \ref{lemma for reversibles} from the language of morphisms to the language of representations. For example, from Proposition  \ref{lemma for reversibles}  we get 
\begin{prop}\label{proposition to be used} 
 Let $(A,\alpha)$ be a reversible system and $(\pi,U)$ its representation. The following conditions are equivalent:
 \begin{itemize}
 \item[i)]  $(\pi,U)$ is covariant,
 \item[ii)] $U^*\pi(a)U=\pi(\alpha_*(a))$ for all $a\in A$,
 \item[iii)] $\overline{\pi}(\overline{\al}_*(1)) \leq U^*U$,
 \item[iv)] $\overline{\pi}(\overline{\al}_*(1))=U^*U$.
       \end{itemize}
\end{prop} 
We also immediately conclude with the following statement, cf. \cite[Theorem 3.7]{kwa-ext}.

\begin{prop}\label{representations vs reversible extensions}
Let $(B,\beta)$ be a natural reversible $J$-extension of   a  $C^*$-dynamical system $(A,\alpha)$ where $J$ is an ideal in $(\ker\alpha)^\bot$. We have a one-to-one correspondence between  representations $(\pi,U)$ of $(A,\alpha)$ covariant on $J$ and covariant representations $(\widetilde{\pi},U)$ of $(B,\beta)$ where
$$
\pi=\widetilde{\pi}|_A, \qquad \widetilde{\pi}\left(\sum_{k=0}^n\beta_*(a_k)\right)=\sum_{k=0}^n U^*\pi(a_k)U, \qquad a_k\in A,\, k=0,...,n.
$$
Under this correspondence $\ker\widetilde{\pi}$ is an invariant ideal in $(B,\beta)$ corresponding to the $J$-pair for $(A,\alpha)$ given by 
$$
I=\ker\pi, \qquad I'=\{ a\in A: U^*U \pi(a)=\pi(a)\}.
$$
In particular, $\widetilde{\pi}$ is faithful if and only if $\pi$ is  faithful and $J=\{ a\in A: U^*U \pi(a)=\pi(a)\}$. 
\end{prop}
\begin{proof}
In view of Remark \ref{from morphisms to representations} it suffices to apply Corollary \ref{corollary for universality} and Theorem \ref{invariant ideals and J-pairs thm}.
\end{proof}

\subsection{Crossed products}\label{subsection crossed products}

We define the relevant crossed products in universal terms. They can be treated as a subclass of crossed products we proposed in \cite[Definition 4.9]{kwa-doplicher} or in \cite[Definition 2.3]{kwa-exel}. If $A$ is unital they coincide with those  constructed and investigated in   \cite{kwa-leb}. 

\begin{defn}\label{crossed product defn}
Let $( A,\alpha)$ be a $C^*$-dynamical system and $J$ an ideal in $ (\ker\alpha)^\bot$. A \emph{crossed product  of $ A$ by $\alpha$ relative to $J$} is   the  $C^*$-algebra $
C^*( A,\alpha,J):=C^*(\iota,u)$ generated by  a  representation $(\iota, u)$ of $( A,\alpha)$ which is universal with respect to  covariance on $J$; i.e. we require  that  if $(\pi,U)$ is a  representation of $(A,\alpha)$ covariant on $J$  then 
$$
\iota(a)\longmapsto \pi(a), \qquad u \longmapsto U
$$
extends  to the homomorphism $\pi\rtimes U$ from $
C^*( A,\alpha,J) $ onto  $C^*(\pi,U)$. If  $J= (\ker\alpha)^\bot$ we write $
C^*( A,\alpha):=C^*( A,\alpha,(\ker\alpha)^\bot)
$ and call it  the (unrelative) \emph{crossed product of $A$ by $\alpha$}.
\end{defn}
\begin{prop}\label{on existence of crossed products}
For any   $C^*$-dynamical system $( A,\alpha)$ and any  ideal $J$ in $ (\ker\alpha)^\bot$, the crossed product $
C^*( A,\alpha,J)$ exists and  is unique (up to natural isomorphism).

 Moreover, 
 \begin{itemize}
 \item[i)] the universal representation $(\iota,u)$ is faithful, $J$-covariant, and we have 
$$
C^*( A,\alpha,J)=\clsp\{u^{*n} \iota(a) u^m: a\in A, n,m\in \N\}, \qquad u \in M(C^*( A,\alpha,J)).
$$
\item[ii)] $C^*( A,\alpha,J)$ sits naturally as an  ideal in $C^*( M(A),\overline{\alpha},\overline{J})$.
\item[iii)] If $(B,\beta)$ is a natural reversible $J$-extension of $(A,\alpha)$ we have natural isomorphism
$$
C^*( A,\alpha,J)\cong C^*(B,\beta),
$$
and $C^*(B,\beta)$ is the closure of a $^*$-algebra consisting of elements of the form
\begin{equation}\label{general form of a guy in cross}
b=\sum_{k=1}^n u^{*n}\iota(b_{-k}) + \iota(b_0) + \sum_{k=1}^n\iota(b_k)u^{n}, \qquad b_{\pm k} \in B, k=0,\pm 1,...,\pm n,
\end{equation}
where  $(\iota,u)$ denotes the universal covariant representation of $(B,\beta)$.
\end{itemize}
\end{prop}
\begin{proof}
Uniqueness of  $C^*(A,\alpha,J)$ follows from universality. The existence can be shown by standard arguments, see \cite{blackadar}. But  we  deduce it, along the way, from the  constructions performed in \cite{kwa-leb} for unital $C^*$-algebras. Namely, by \cite[Theorem 1.11, Proposition 3.7]{kwa-leb} we know that the crossed product $C^*( M(A),\overline{\alpha},\overline{J})$ exists and is generated by a faithful $\overline{J}$-covariant representation $(\overline{\iota},u)$ of  $(M(A),\overline{\alpha})$ on a certain Hilbert space $H$. Let $\iota:=\overline{\iota}|_{A}$.  We claim that, modulo non-degeneracy issues to be explained below, the pair $(\iota,u)$ is the universal $J$-covariant representation of $(A,\alpha)$.

Indeed, by \cite{kwa-leb} we know that $C^*( M(A),\overline{\alpha},\overline{J})=\clsp\{u^{*n} \overline{\iota}(a) u^m: a\in M(A), n,m\in \N\}$.  Consider  the self-adjoint  Banach space  $C^*( A,\alpha,J):=\clsp\{u^{*n} \overline{\iota}(a) u^m: a\in A, n,m\in \N\}$.  Let $a, b\in M(A)$ and $k,l,m, n\in \N$, where for instance $m\leq l$. Calculations similar to that in the proof of Proposition \ref{iteration of representations} give
$$
\left(u^{*k} \overline{\iota}(a) u^l\right) \left(u^{*m} \overline{\iota}(b) u^n\right)=u^{*k} \overline{\iota}(a\overline{\alpha}^{l-m}(\overline{\alpha}^m(1)b)) u^{n+l-m}.
$$
Hence if either $a$ or $b$ is in $A$ the product  is in $u^{*k} \iota(A) u^{n+l-m}$. This shows that $C^*( A,\alpha,J)$ is an ideal in $C^*( M(A),\overline{\alpha},\overline{J})$. In particular,  $C^*( A,\alpha,J)$ is a $C^*$-algebra  generated   by $\iota(A)\cup \iota(A)u$, and  the orthogonal projection $P$ from $H$  onto the essential space $C^*( A,\alpha,J)H$ for $C^*( A,\alpha,J)$ commutes with elements of $C^*( M(A),\overline{\alpha},\overline{J})$. Significantly, $P$ commutes with $u$. Furthermore, since $u^{*k} \iota(A)=u^{*k} \overline{\iota}(\overline{\alpha}^k(1))\iota(A)=u^{*k} \iota(\alpha^k(A)A)=\iota(A)u^{*k} \iota(A)$ we see that $C^*( A,\alpha,J)$ contains  $\iota(A)$ as a non-degenerate $C^*$-algebra. Thus $\iota(A)H=PH$ and  $\iota(a)=\iota(a)P$ for all $a\in A$. Accordingly, restricting $\iota$ and $u$ to $PH$ we see that $(\iota,u)$ is a faithful representation of $(A,\alpha)$ on $PH$. It is $J$-covariant because 
$$
\{a\in A: u^*u\iota(a)P=\iota(a)P\}=\{a\in A: u^*u\iota(a)=\iota(a)\}=A \cap \overline{J}=J, 
$$ 
cf. Lemma \ref{diamonds lemma}. Universality of $(\iota,u)$  follows  from  Lemma \ref{covariance for extensions} i),  cf. Remark \ref{from morphisms to representations}, and the universality of  $C^*( M(A),\overline{\alpha},\overline{J})$.  Operator $u$ restricted to $PH$ identifies with a multiplier of $C^*( A,\alpha,J)$ in an obvious way. This shows parts i) and ii).

Part iii) follows  from Proposition \ref{representations vs reversible extensions} and  part i) of the present assertion, plus  a simple observation that for $b\in B$ we have
$$
u^{*m}(b)u^n=
\begin{cases}
 u^{*(m-n)}\iota(\beta_*^n(b)), & \textrm{ if } m \geq n, \\
 \iota(\beta_*^m(b))u^{n-m}, & \textrm{ if } m < n.
\end{cases}
$$
\end{proof}
\begin{rem}
In the sequel, we will  assume the identification
$$
C^*( A,\alpha,J)=C^*(B,\beta)
$$ 
where  $(B,\beta)$ is the natural reversible $J$-extension of $(A,\alpha)$. 
Moreover,  we will write  $(\iota,u)$ for both  the universal covariant representation of $(B,\alpha)$ and  the universal $J$-covariant representation of $(A,\alpha)$.
\end{rem}
\begin{rem}\label{remark to mentioned in the introduction}
One can  define  the crossed product $C^*( A,\alpha,J)$ for an arbitrary  ideal $J$ in $A$. Just let $C^*( A,\alpha,J)=C^*(\iota(A),u)$ be the $C^*$-algebra generated by    a  representation $(\iota, u)$  of $( A,\alpha)$ universal with respect to representations satisfying \eqref{covariance rel3}. 
However, if $J\nsubseteq (\ker\alpha)^\bot$ the universal representation $(\iota, u)$ is not faithful. Moreover, by the reduction procedure described in \cite[Section 5.3]{kwa-leb}, see also  \cite[Example 6.24]{kwa-doplicher},
$$
R:=\{a \in  A:  \alpha^n(a)\in J \textrm{ for all } n\in \N  \textrm{ and } \lim_{n\to \infty} \alpha^n(a)=0\}
$$
is the smallest positively invariant ideal in $(A,\alpha)$ such that $q_R(J) \subseteq (\ker\alpha_R)^\bot$, and we have a natural isomorphism
$$
C^*( A,\alpha,J)\cong C^*( A/R,\alpha_R,q_R(J)). 
$$
Thus this seemingly more general situation reduces  easily  to that  of Definition \ref{crossed product defn}. For the particular case of Stacey's crossed product $A\times_\alpha^1 \N$ we have 
$$
A\times_\alpha^1 \N=C^*( A,\alpha,A)\cong C^*( A/R,\alpha_R) \quad \textrm{where} \quad R=\{a \in  A:  \lim_{n\to \infty} \alpha^n(a)=0\}.
$$
Here  $( A/R,\alpha_R)$ can be regarded as the largest subsystem of $(A,\alpha)$ such that $\alpha_R$ is a monomorphism.
\end{rem}
Let now $(A,\alpha)$ and $J\subseteq  (\ker\alpha)^\bot$ be fixed.  By  universality of the representation $(\iota,u)$ generating  $C^*( A,\alpha,J)$ we have a (point-wise continuous) circle action $\mathbb{T}=\{z\in \C: |z|=1\}\ni z \longmapsto \gamma_z \in \Aut(C^*( A,\alpha,J))$ determined by 
$$
\gamma_z(\iota(a))=\iota(a), \qquad \gamma_z(u)=z u, \qquad a\in A, \,\,z\in \mathbb{T}.
$$
We denote this action by $\gamma$ and refer to it as the \emph{gauge action} on $C^*( A,\alpha,J)$.
One infers from Proposition  \ref{on existence of crossed products} that the algebra $\iota(B)$, where  $(B,\beta)$ is the natural reversible $J$-extension of $(A,\alpha)$,  is the fixed point algebra for the gauge action $\gamma$. Integration over the Haar measure $\mu$ on $\mathbb{T}$ gives  a conditional expectation
$$
E(a)=\int_{\mathbb{T}} \gamma_z(a) d\mu, \qquad a\in C^*( A,\alpha,J),
$$  
from $C^*( A,\alpha,J)$ onto  $\iota(B)$. We get the following  version of a standard result  often called gauge-invariant uniqueness theorem. 
\begin{prop}\label{gauge-invariance equivalence}
Let $(\pi,U)$ be a faithful $J$-covariant representation. The following conditions are equivalent:
\begin{itemize}
\item[i)] $\pi\rtimes U: C^*( A,\alpha,J)\to C^*(\pi,U)$ is an isomorphism,
\item[ii)] we have the equality $J=\{ a\in A: U^*U \pi(a)=\pi(a)\}$ and  there is a circle action $\gamma^{\pi}$ on  $C^*(\pi,U)$ such that 
$$
\gamma_z^\pi(\pi(a))=\pi(a), \qquad \gamma_z^\pi(U)=z U\qquad a\in A, \,\,z\in \mathbb{T},
$$  
\item[iii)] we have the equality $J=\{ a\in A: U^*U \pi(a)=\pi(a)\}$ and there is a conditional expectation from $C^*( \pi,U)$ onto $\sum_{n=0}^\infty U^{*n} \pi(A) U^n$, sending the spaces $U^{*n} \pi(A) U^m$ to zero when $m\neq n$.
\end{itemize}
\end{prop}
\begin{proof}
Implications i) $\Rightarrow$ ii) $\Rightarrow$ iii) follow easily from the discussion above and the last part of Proposition \ref{representations vs reversible extensions}.  To show iii) $\Rightarrow$ i) assume that $E^\pi$ is a  conditional expectation from $C^*( \pi,U)$ onto $\sum_{n=0}^\infty U^{*n} \pi(A) U^n$, as  described above. By the last part of Proposition \ref{representations vs reversible extensions},  $\pi\rtimes U|_{B}=\widetilde{\pi}\circ \iota$ is a faithful representation of $B$. 
By part iii) of Proposition \ref{on existence of crossed products} it suffices to show that $\pi\rtimes U(b)\neq 0$   if  $b$ is a non-zero element  given by \eqref{general form of a guy in cross}. But if $b\neq 0$ is given by \eqref{general form of a guy in cross}, then $u^{*k} \iota(b_{-k})\neq 0$ or $ \iota(b_{k})u^k\neq 0$ for some $k=0, \pm 1,..., \pm n$. Fix such $k$. Since
$$
0< u^{*k} \iota(b_{-k})  \iota(b_{-k})^*u^{k} + \iota(b_{k})u^ku^{*k}\iota(b_{k^*})= \iota(\beta_*^k(b_{-k} b_{-k}^*) + b_k\overline{\beta}_*^k(1)b_k^*)\in \iota(B)
$$
we get  
$$
0< \widetilde{\pi}\left( \iota(\beta_*^k(b_{-k} b_{-k}^*) + b_k\overline{\beta}_*^k(1)b_k^*)\right)\leq E^\pi\left( \pi\rtimes U(b b^*)\right).
$$ 
This implies that  $\pi\rtimes U(b)\neq 0$. 
\end{proof}

\subsection{Gauge-invariant ideals}

Representations $\pi\rtimes U$ of $C^*( A,\alpha,J)$ satisfying condition ii) in Proposition \ref{gauge-invariance equivalence} are sometimes called \emph{gauge-invariant representations}. Kernels of such representations are  \emph{gauge-invariant ideals} in $C^*( A,\alpha,J)$, i.e. ideals that are invariant under the gauge action on $C^*( A,\alpha,J)$.  Clearly, such ideals form a lattice. 
We describe this lattice in two steps.
\begin{prop}\label{ideals in a reversible case} Let $(B,\beta)$ be a reversible $C^*$-dynamical system. Relations
\begin{equation}\label{from ideals to invariant ideals}
 \I=\{b \in B: \iota(b) \in \II\},
 \end{equation}
\begin{equation}\label{gauge-invariant ideal from J-pair}
  \II = \sum_{k=1}^{\infty} u^{*k}\iota(\I) + \iota(\I)+ \sum_{k=1}^{\infty} \iota(\I) u^{k}, 
\end{equation}
establish an order isomorphism between the  lattices of   gauge-invariant ideals $\II$  in $C^*(B,\beta)$  and   invariant ideals $\I$ in $(B,\beta)$. Moreover, under the above correspondence  we have
\begin{equation}\label{quotient isomorphism on B level}
C^*(B,\beta)/\II\cong C^*(B/\I,\beta_{\I}).
\end{equation}
 \end{prop}
\begin{proof}
Let $\I$ be an invariant ideal in $(B,\beta)$ and define  $\II$ by \eqref{gauge-invariant ideal from J-pair}.
To see that $\II$ is an ideal  in $C^*( B,\beta)$ recall that the projections $u^{*k}u^k$, $k\in \N$, commute with elements of $\iota(B)$, and $u^*\iota(b)u=\iota(\beta_*(b))$, $b\in B$, see Propositions \ref{proposition to be used}, \ref{lemma for extensions of transfers}. Hence for  $a,b\in B$ and $m\geq n$, we get
\begin{align*}
\iota(a)u^m  \cdot u^{*n}\iota(b)&= \iota\left(a \beta^{m-n}\big(\overline{\beta}^{n}(1)b\big)\right)u^{m-n}, 
\\
u^{*n}\iota(b) \cdot \iota(a)u^m &= \iota\big(\beta_*^{n}(ba )\big)u^{m-n}.
\end{align*}
Thus (by invariance of $\I$) if either $a$ or $b$ is in $\I$ then both of the above products are in $\II$.  By passing to adjoints,  we arrive at a similar conclusion for $m <n$. Accordingly, in view of Proposition \ref{on existence of crossed products} iii), $\II$ is an ideal in $C^*( B,\beta)$. Clearly, $\II$  is gauge-invariant and \eqref{from ideals to invariant ideals} holds. To show \eqref{quotient isomorphism on B level}  denote by $(\iota_{\I}, u_{\I})$ the universal representation generating the crossed product  $C^*(B/\I,\beta_{\I})$ and note that $(\iota_{\I}\circ q_{\I}, u_{\I})$ is a covariant representation of $(B,\beta)$. Therefore we have the epimorphism 
$$
C^*(B,\beta)\stackrel{\Psi}{\longmapsto}  C^*(B/\I,\beta_{\I})\quad\textrm{where}\quad
\Psi(\iota(b))=\iota_{\I}( q_{\I}(b)), \,\,\, \Psi(u)=u_{\I}.
$$
Since $\II\subseteq \ker\Psi$ this mapping factors through to an epimorphism
$$
C^*(B,\beta)/\II\stackrel{\Psi_{\II}}{\longmapsto}  C^*(B/\I,\beta_{\I}) \quad\textrm{where}\quad \Psi_{\II}\circ q_{\II}=\Psi.
$$
As $\II$ is gauge-invariant, the gauge action on $C^*(B,\beta)$ factors through to  $C^*(B,\beta)/\II$. With this  circle  action on $C^*(B,\beta)/\II$, $\Psi_{\II}$ becomes a gauge-invariant epimorphism which is injective on the fixed point algebra $q_{\II}(\iota(B))$. Hence  by the general gauge-invariant uniqueness theorem, cf., for instance, \cite[Theorem 4.5.1]{bo},  $\Psi_{\II}$ is  is injective on $C^*(B,\beta)/\II$. This proves \eqref{quotient isomorphism on B level}. 
 
Let now $\II'$ be an arbitrary  ideal in $C^*(B,\beta)$ and put $\I:=\{b \in B: \iota(b) \in \II'\}$. Then $\iota(\beta(\I))=u\iota(\I)u^*\subseteq \II'$ and  $\iota(\beta_*(\I))=u^*\iota(\I)u\subseteq \II'$. Thus $\I$ is invariant in $(B,\beta)$ by Lemma \ref{helping lemma}. Plainly, the gauge-invariant ideal   $\II$ given by \eqref{gauge-invariant ideal from J-pair} is contained in  $\II'$. Accordingly, the identity factors through to the epimorphism
$$
C^*(B,\beta)/\II \stackrel{\Phi}{\longmapsto}  C^*(B,\beta)/\II'.
$$
If $\II'$ is gauge-invariant then both of the algebras $C^*(B,\beta)/\II$ and $C^*(B,\beta)/\II'$ are equipped with circle actions induced from $C^*(B,\beta)$. Epimorphism $\Phi$ is gauge-invariant with respect to these actions and $\Phi$ is injective on the fixed point algebra of the circle action on $C^*(B,\beta)/\II$. Hence  $\Phi$ is an isomorphism, again by \cite[Theorem 4.5.1]{bo}. Consequently, $\II=\II'$.
\end{proof}
\begin{rem}
The last part of  the  proof above shows that relation \eqref{from ideals to invariant ideals}
 establishes an order preserving surjection from the lattice of all ideals in $C^*(B,\beta)$ onto the lattice of invariant ideals in $(B,\beta)$. Thus the lattice of gauge-invariant ideals in  $C^*(B,\beta)$ can be regarded as \emph{an order-retract} of the lattice of all ideals.
\end{rem}
Combining the above statement with Theorem \ref{invariant ideals and J-pairs thm} we obtain a general result.
\begin{thm}\label{lattices descriptions main thm}
Let $(A,\alpha)$ be a $C^*$-dynamical system and $J$ an ideal in $(\ker\alpha)^\bot$. We have an order isomorphism  from  the lattice of gauge-invariant  ideals $\II$ in $C^*( A,\alpha,J)$ onto the lattice of  $J$-pairs $(I,I')$ for $(A,\al)$. It is  determined by relations
\begin{equation}\label{J-pair from gauge-invariant ideal}
I=\{a \in A: \iota(a) \in \II\}, \qquad I'=\{a\in A: (1-u^*u) \iota(a) \in \II\},
\end{equation}
$$
\II \textrm{ is generated by }  \iota(I) + (1-u^*u)\iota(I').
$$ 
For ideals  satisfying the above relations  we have
$$
C^*(A,\alpha,J)/ \II \cong C^*(A/I,\alpha_I,q_I(I')).
$$
Restricting the above order isomorphism we get an order preserving bijective correspondence between ideals $\II$ in $C^*(A,\alpha,J)$ which are generated by their intersection with $\iota(A)$,  and $J$-invariant ideals $I$ in $(A,\alpha)$. This restricted correspondence  is  determined by relations
\begin{equation}\label{J-invariant ideals and A-generated ideals}
I=\{a \in A: \iota(a) \in \II\}, \qquad  \II=\clsp\{u^{*m}\iota(a)u^{n}: a\in I, n,m \in \N\},
\end{equation}
and  for the corresponding ideals we have
$$
C^*(A,\alpha,J)/ \II \cong C^*(A/I,\alpha_I,q_I(J)).
$$
\end{thm}
\begin{proof}
We obtain the desired order isomorphism by  composing the  isomorphisms from Proposition \ref{ideals in a reversible case} and Theorem \ref{invariant ideals and J-pairs thm}. Moreover, using (twice) the isomorphism from Proposition \ref{on existence of crossed products}  iii), the equivalence of $C^*$-dynamical systems from Theorem \ref{invariant ideals and J-pairs thm} and the isomorphism \eqref{quotient isomorphism on B level} we get
\begin{align*}
C^*(A,\alpha,J)/ \II &\cong C^*(B,\beta)/ \II \cong  C^*(B/\I,\beta_{\I})\cong C^*(B_{(I,I')},\beta_{(I,I')})
\\
& \cong C^*(A/I,\alpha_I,q_I(I')).
\end{align*}
This proves the first part of the assertion.
\\
For the second part  note that the ideal  $\II$ in $C^*(A,\alpha,J)$ generated by $I:=\II \cap \iota(A)$ corresponds, via the above isomorphism, to the $J$-pair $(I,I+J)$, see also Remark \ref{remark on A-generated ideal in B}. 
\end{proof}
\begin{cor}
Suppose $\ker\al$ is a complemented ideal in $A$.  Relations \eqref{J-invariant ideals and A-generated ideals}
establish an order isomorphism between the lattices of  gauge-invariant  ideals $\II$ in $C^*( A,\alpha)$  and invariant ideals $\I$ in $(A,\al)$. Under this correspondence we have
$$
C^*(A,\alpha)/\II\cong C^*(A/I,\alpha_{I}).
$$
In particular, in this case every gauge-invariant ideal $\II$ in  $C^*(A,\alpha)$ is generated by its intersection with $\iota(A)$.
\end{cor}
\begin{proof}
Apply the fact that any $(\ker\alpha)^\bot$-pair in $(A,\alpha)$   is of the form $(I, I+(\ker\alpha)^\bot)$ where $I$ is invariant in $(A,\alpha)$, cf. Remark \ref{remark on J-pairs}. In particular, we have $q_I((\ker\alpha)^\bot)=(\ker\alpha_I)^\bot$.
\end{proof}
\subsection{Simplicity}
Now we use the above  result to study simplicity of $C^*(A,\alpha,J)$. We will juxtapose the  conditions we get with \cite[Theorem 4.1]{Szwajcar}, \cite[Corollary 1]{Ortega_Pardo} which provide    certain sufficient conditions for simplicity of $C^*(A,\alpha)$ when $\alpha$ is a monomorphism. To deal with non-injective $\alpha$ we will  use Corollary \ref{cor for freeness of quasinilpotents} from the next subsection, and a  notion of pointwise quasinilpotence, which seems to be a novelty in the present context.% in the literature.
\begin{defn}\label{minimality definitions}
 We say that  $\alpha$ is \emph{ minimal} provided that there are no nontrivial invariant  ideals  in $(A,\alpha)$. We call $\alpha$ \emph{pointwise quasinilpotent} if $\lim_{n\to \infty}\alpha^n(a)=0$, for all $a\in A$. A monomorphism  $\alpha$ is called \emph{inner} if  there is an isometry $v\in M(A)$ such that $\alpha(a)=v a v^*$, $a\in A$. 
\end{defn}
Note that if  $\alpha$ has no   positively invariant ideals then $\alpha$ is necessarily a minimal monomorphism. Moreover, if $\alpha$ is a monomorphism and $A$ is unital, it can be shown that $\alpha$  is minimal if and only if  there are no positively invariant  ideals  in $(A,\alpha)$, cf. \cite{Szwajcar}, \cite{Murphy2}.

 \begin{thm}\label{theorem on simplicity}
 If the  algebra $C^*(A,\alpha,J)$  is  simple then $J=(\ker\alpha)^\bot$,   $\alpha$ is  minimal and either $\alpha$ is pointwise quasinilpotent or $\alpha$ is a monomorphism and no power $\alpha^n$, $n>0$, is inner. 
 
 If $\alpha$ is minimal each of the following conditions imply that $C^*(A,\alpha)$ is simple:
 \begin{itemize}
 \item[i)]  $\alpha$ is  pointwise quasinilpotent,
\item[ii)]  $\alpha$ is injective,  $A$ is unital, and no power $\alpha^n$, $n>0$, is inner,
\item[iii)] $\alpha$ is injective with hereditary range, $A$ is separable, and no power $\alpha^n$, $n>0$, is inner.
\end{itemize}
\end{thm}

\begin{proof}
 As  $(\{0\}, J)$ and $(\{0\},(\ker\alpha)^\bot)$ are always  $J$-pairs,  Theorem \ref{lattices descriptions main thm} implies  that $C^*(A,\alpha,J)$  is not simple unless $J=(\ker\alpha)^\bot$. Similarly,   $C^*(A,\alpha)$ is not  simple unless there are no invariant ideals in $(A,\alpha)$.
Moreover,  one readily sees that 
$$
I:=\{a\in A: \lim_{n\to \infty} \alpha^n(a)=0\}
$$
is an invariant ideal in $(A,\alpha)$. Thus if $C^*(A,\alpha)$ is simple then $\alpha$ is minimal and either $I=A$ or $I=\{0\}$, that is either $\alpha$ is pointwise quasinilpotent or $\alpha$ is a monomorphism. 

Suppose then that $\alpha$ is a monomorphism and   for some  $n>0$, $\alpha^n$ is inner so that  $\alpha^n(a)=v a v^*$, for all $a\in A$, where   $v\in M(A)$ is an isometry. We claim that we may reduce our considerations to the case  $(A,\alpha)$ is a reversible $C^*$-dynamical system. Indeed, if $(B,\beta)$ is the natural reversible extension of $(A,\alpha)$ then the natural reversible extension of $(A,\alpha^n)$ can be identified with $(B,\beta^n)$. But  $(id, v)$ is a faithful covariant representation of $(A,\alpha^n)$ in $A$, cf. Remark \ref{remark for one proof and for Bialorusy}. Thus by Proposition \ref{representations vs reversible extensions} we have the isomorphism $B\cong \sum_{k=0}^{\infty}v^{*k}Av^{k}
=A$ under which  $\beta^n(b)=vbv^*$, $b\in B$, and $v\in M(B)$. This proves our claim. Let us then assume that $(A,\alpha)$ is reversible. In view of  Proposition \ref{proposition to be used} we have $\overline{\alpha}^n_*(a)=v^*a v$ for all $a\in M(A)$. In particular  $\overline{\alpha}^n_*(v)=v^*v v=v$.
 For each $m=1,...,n$ we put $v_m:=\overline{\alpha}^m_*(v)$. Note that for  $a\in A$ we have
$$
v_m a=\overline{\alpha}^m_*(v) a= \overline{\alpha}^m_*(v \alpha^m(a))= \overline{\alpha}^m_*( \alpha^{m+n}(a)v)=\alpha^{n}(a) \overline{\alpha}^m_*(v)=\alpha^{n}(a) v_m.
$$
Since $ \overline{\alpha}^n(1) v=v$, all the more  $ \overline{\alpha}^m(1) v=v$ and thus we get
$$
v_m^* v_m=\overline{\alpha}^m_*\left(v^*\overline{\alpha}^m(\overline{\alpha}^m_*(v))\right)= \overline{\alpha}^m_*(v^* \overline{\alpha}^m(1) v\overline{\alpha}^m(1))=\overline{\alpha}^m_*(v^*  v)= \overline{\alpha}^m_*(1)=1. 
$$
Thus $v_m\in M(A)$ are isometries such that $\alpha^{n}(a)=v_m a v_m^*$ for all $a\in A$. These isometries  commute. Indeed,  $\overline{\alpha}^n(v_k)=v_k v_k v_k^*=v_k \overline{\alpha}^n(1)=v_k v_m v_m^*$ and $\overline{\alpha}^n(v_k)=v_m v_k v_m^*$ imply that $v_kv_m=v_k v_m$, for $m,k=1,...,n$. Moreover,  for $m=1,...,n$ (and $v_0=v$) we have  $\overline{\alpha}(v_m)=\overline{\alpha}(1)v_{m-1}\overline{\alpha}(1)=v_{m-1}\overline{\alpha}(1)$. Thus putting 
$$
w:=v_0v_1...v_{n-1}
$$
we get an isometry in $M(A)$ such that
\begin{equation}\label{relacje to be used}
waw^*=\alpha^{n^2}(a) \quad a\in A,\quad \textrm{ and } \quad \overline{\alpha}(w)=w\overline{\alpha}(1).
\end{equation}
Now assume that $A$ is faithfully represented as a $C^*$-subalgebra of $\B(H)$ acting non-degenerately on $H$. Put $\H:=H\oplus \overline{\alpha}(1)H \oplus ... \oplus  \overline{\alpha}^{n^2-1}(1)H$. Using \eqref{relacje to be used} one readily checks that the formulae
$$
\pi(a):=\textrm{diag}(a, \alpha(a),...,\alpha^{n^2-1}(a)), \quad 
U=\left(\begin{array}{ccccc} 
0 &  \overline{\alpha}(1) & 0 &... & 0\\
0 &  0 & \overline{\alpha}^2(1) &... & 0\\
\vdots & \vdots & \vdots & \ddots & \vdots\\
0 &  0 & 0 &... & \overline{\alpha}^{n^2-1}(1)\\
w   &  0 & 0 &... & 0\\
\end{array}\right)
$$ 
define a covariant representation $(\pi, U)$ of $(A,\alpha)$ on $\H$. Moreover, $U^{n^2}=\overline{\pi}(w)$ where $\overline{\pi}$ is the extension of $\pi$ onto $M(A)$. Hence $(\pi, U)$ gives rise to a representation $\pi\rtimes U$ of $C^*(A,\alpha)$ which is faithful on $\iota(A)$ but not globally faithful because for any $a \in A$ we have
$$
(\pi\rtimes U)( \iota(wa)- u^{n^2}\iota(a))=0.
$$
This contradicts the simplicity of $C^*(A,\alpha)$.

Part i)  follows from Theorem \ref{lattices descriptions main thm} and Corollary \ref{cor for freeness of quasinilpotents} (which we prove in next section).
Part ii)  follows from \cite[Theorem 4.1]{Szwajcar}, and part iii) follows from \cite[Corollary 1]{Ortega_Pardo} (note that  the proof of \cite[Theorem 1]{Ortega_Pardo} is based on results of \cite{OlesenPedersen3} which were proved under assumption that $A$ is separable and this assumption seems to be essential).
\end{proof}
\subsection{Topological freeness for reversible $C^*$-dynamical systems}\label{Topological freeness}
Throughout this subsection $(B,\beta)$ stands for a reversible $C^*$-dynamical system. One can think of $(B,\beta)$ as of the natural $J$-reversible extension of a certain (not necessarily reversible) $C^*$-dynamical system $(A,\alpha)$. We will adopt this viewpoint in Corollary \ref{cor for freeness of quasinilpotents} below. 

Let $\widehat{B}$ be the spectrum of $B$ equipped with the Jacobson topology. As it is generally accepted, we will abuse notation and use $\pi$ to denote both   an irreducible representation  of $B$  and its equivalence class  in $\widehat{B}$. Since $\beta_*(B)=(\ker\beta)^\bot$ is an  ideal in $B$ and $\beta(B)=\overline{\beta}(1)B\overline{\beta}(1)$ is a hereditary subalgebra of $B$ we have  natural identifications  of spectra of $\beta_*(B)$ and $\beta(B)$ with open subsets of $\widehat{B}$: 
$$
\widehat{\beta_*(B)}=\{\pi \in \widehat{B}: \pi(\beta_*(B))\neq 0\},\qquad \widehat{\beta(B)}=\{\pi \in \widehat{B}: \pi(\beta(B))\neq 0\}. 
$$
With these identifications the homeomorphisms dual  to the mutually inverse  isomorphisms $
\beta: \beta_*(B) \to \beta(B)$  and $
\beta_*: \beta(B) \to \beta_*(B)$ yield   partial homeomorphisms of $\widehat{B}$:
$$
\widehat{\beta}: \widehat{\beta(B)} \to \widehat{\beta_*(B) }, \qquad \widehat{\beta}^{-1}=\widehat{\beta_*}: \widehat{\beta_*(B)} \to \widehat{\beta(B) },
$$
cf. also \cite{kwa-leb}, \cite{kwa-interact}. 
More precisely, let $\pi: B\to \B(H)$  be  a representation. If  $\pi \in \widehat{\beta_*(B)}$ then  $\widehat{\beta_*}(\pi)$ is the unique (up to unitary equivalence)  irreducible extension of $\pi\circ \beta_*: \beta(B)\to \B(H)$ up to  irreducible representation $\widehat{\beta_*}(\pi):B\to \B(\H)$ where
 $\widehat{\beta_*}(\pi)(\beta(B))\H=H\subseteq \H$. If  $\pi \in \widehat{\beta(B)}$, then $\widehat{\beta}(\pi)$ is the unique  extension of $\pi\circ \beta: \beta_*(B)\to \B(\beta(B)H)$ up to the (irreducible) representation $\widehat{\beta}(\pi):B\to \B(\beta(B)H)$.
\begin{defn}\label{topolo disco}
We say that a partial homeomorphism $\varphi$ of a topological space $X$   is   {\em
topologically free} if  the set of its periodic points of any given period  $n>0$  has empty interior. 
\end{defn}

The following statements could be deduced from the  results of \cite{kwa}, in the same fashion as \cite[Theorem. 2.19]{kwa-interact}. However, we can omit the use of Hilbert bimodules  by reducing the problem to the unital case  treated  in \cite[Theorem 2.19]{kwa-interact}, cf. also \cite[Theorem 3.1]{kwa-leb}. 
\begin{thm}\label{topological freeness for reversible systems theorem}
 Let $(B,\beta)$ be a reversible $C^*$-dynamical system. If the dual   partial homeomorphism $\widehat{\beta}$  is topologically free then
 for every faithful covariant representation $(\pi, U)$ of $(B,\beta)$,  $\pi\rtimes U: C^*( B,\beta)\to C^*(\pi,U)$ is an isomorphism. 
\end{thm}
\begin{proof}
The extended system $(M(B),\overline{\beta})$ is reversible, see Proposition \ref{lemma for extensions of transfers}. With natural identifications  $\widehat{B}$ is an open and dense subset of $\widehat{M(B)}$ and $\widehat{\beta}$ is restriction of the  partial homeomorphism $\widehat{\overline{\beta}}$ of $\widehat{M(B)}$ dual to $\overline{\beta}$. Thus   topological freeness of $\overline{\beta}$ is equivalent to topological freeness of  $\widehat{\overline{\beta}}$. Moreover, by  Lemma \ref{covariance for extensions} iii), cf. Remark \ref{from morphisms to representations},  $(\overline{\pi},U)$ is a faithful covariant representation of $(M(B),\overline{\beta})$. Hence by \cite[Theorem 2.19 i)]{kwa-interact},  $\overline{\pi}\rtimes U: C^*( M(B),\overline{\beta})\to C^*(\overline{\pi},U)$ is an isomorphism which (with  the identification from Proposition \ref{on existence of crossed products} ii)) restricts to the isomorphism $\pi\rtimes U: C^*( B,\beta)\to C^*(\pi,U)$.
\end{proof}
\begin{rem}\label{remark on pederesen olesen} There are reasons to believe  that  topological freeness and  the isomorphism property described in the assertion of Theorem \ref{topological freeness for reversible systems theorem} are  equivalent,  at least for a large class of $C^*$-algebras. For instance, by \cite[Theorem 10.4]{OlesenPedersen3} this is the case when $\beta$ is an automorphism and $B$ is separable.  We adapt a standard construction to  show that this is  true  when $B$ is commutative.
\end{rem}
  \begin{ex}[Orbit representation]\label{orbit representation} Let  $B$ be commutative. By an orbit in $(\widehat{B}, \widehat{\beta})$ we mean a maximal set  $\OO\subseteq \widehat{B}$ whose elements can be indexed by  $\Z$, $\Z_+$, $\Z_{-}$, or $\{1...,n\}$, for $n >0 $, in such a way that $\widehat{\beta}(x_k)=x_{k+1}$, if  $x_k$, $x_{k+1}\in \OO$, and $\widehat{\beta}(x_n)=x_{1}$ in the case $\OO$ is a periodic orbit with minimal period $n$.  To each $\OO$ we attach a covariant representation $(\pi_{\OO},U_{\OO})$ of  $(B,\beta)$  on the Hilbert space $\ell^2(\OO)$ with orthonormal basis $\delta_{x_k}$, $x_k\in \OO$. We let $\pi_{\OO}$ to be the diagonal representation given by $\pi_{\OO}(a) \delta_{x_k} = x_k(a)\delta_{x_k}$, $x_k\in \OO$. We define  $U$ as the shift
$$
 U_{\OO} \delta_{x_k}=\delta_{x_{k-1}} \textrm{ if } x_{k-1} \in \OO,\quad \textrm{ and }\quad U_{\OO} \delta_{x_k}=0 \textrm{  if } x_{k-1} \notin \OO, 
$$
 unless $\OO=\{x_1,...,x_n\}$ 
is periodic and $k=1$ in which case we put $U_{\OO} \delta_{x_1}=x_n$. The direct sums
$$
\pi:= \bigoplus \pi_{\OO},\qquad U:=\bigoplus U_{\OO}
$$  
over all orbits in $(\widehat{B}, \widehat{\beta})$ define a faithful covariant representation of $(B, \beta)$. 
Suppose now that $\widehat{\beta}$ is not topologically free. Hence there is a nonempty  open set $U$ consisting of periodic points of period $n$. Then for any non-zero element $b\in \bigcap_{x\in \widehat{B}\setminus U}\ker x$  we have 
$$
0 \neq b- bu^{n} \in \ker(\pi\rtimes U).
$$
Hence $\pi\rtimes U$ is not faithful. In particular, $\ker(\pi\rtimes U)$ is not gauge-invariant, cf. Proposition \ref{gauge-invariance equivalence}.
\end{ex} 
 
Since invariant ideals in $(B,\beta)$ correspond to  open sets invariant under  $\widehat{\beta}$,  Theorem \ref{topological freeness for reversible systems theorem} leads us  to a condition implying that all ideals in $C^*(B,\beta)$ are gauge-invariant. 

\begin{defn}\label{phi invariance}
Let $\Delta$  be a domain of a  partial homeomorphism $\varphi$ of $X$.   A set $V\subseteq X$ is $\varphi$-\emph{invariant}    if 
$
\varphi(V\cap \Delta)= V\cap \varphi(\Delta). 
$
We say that $\varphi$ is (residually) \emph{free}, if it is topologically free when restricted to any  closed invariant set.
\end{defn}
\begin{cor}\label{freeness implies gauge-invariance}
 If $\widehat{\beta}$  is free then  all ideals in $C^*(B,\beta)$ are gauge-invariant and thus  we have natural isomorphisms  between the lattices of the following objects: 
 \begin{itemize}
 \item[i)] ideals in $C^*(B,\beta)$,
 \item[ii)] invariant ideals in $(B,\beta)$,
 \item[iii)] open $\widehat{\beta}$-invariant subsets of $\beta$.
 \end{itemize}
Moreover, if $B$ is commutative then  $\widehat{\beta}$  is free  if and only if  all ideals in $C^*(B,\beta)$ are gauge-invariant.
 \end{cor}
 \begin{proof}
 Let  $\II'$ be   an ideal in $C^*(B,\beta)$ and denote by $\II$ the ideal in $C^*(B,\beta)$ generated by $\iota(B)\cap \II'$. Recalling the last part of the proof of  Proposition  \ref{ideals in a reversible case},  we  have the epimorphism
$$
C^*(B/\I,\beta_{\I})\cong C^*(B,\beta)/\II \stackrel{\Phi}{\longmapsto}  C^*(B,\beta)/\II' 
$$
which is injective on the image of  $B/\I$. As $\I$ is invariant,  the spectrum of $B/\I$ is identified with the $\widehat{\beta}$-invariant closed set $\widehat{B}\setminus \widehat{\I}$ and  then $\widehat{\beta}_{\I}$ becomes a restriction of $\widehat{\beta}$. Thus if $\widehat{\beta}$ is free then $(\widehat{B/\I},\widehat{\beta_{\I}})$ is topologically free. Therefore,  by Theorem \ref{topological freeness for reversible systems theorem}, $\Phi$ is an isomorphism and consequently $\II'=\II$ is gauge-invariant. 
\\
Assume  that $B$ is commutative and  $\widehat{\beta}$ is not free. Then there is an invariant ideal $\I$ in $(B,\beta)$ such that the  partial homeomorphism  $\widehat{\beta_{\I}}$ is not topologically free.   By Example \ref{orbit representation} there is an ideal $\II'$ in $C^*(B/\I,\beta_{\I})$, which is not gauge-invariant. Denoting by $\II$ the (gauge-invariant) ideal in $C^*(B,\beta)$ generated by $\iota(\I)$ and identifying  $C^*(B/\I,\beta_{\I})$ with $ C^*(B,\beta)/\II$ we get an ideal  $q_{\II}^{-1}(\II')$  in $C^*(B,\beta)$ which fails to be gauge-invariant.
 \end{proof}

 \begin{cor}\label{cor for simplicity} 
  If $\beta$ is a minimal monomorphism with hereditary range  and $\widehat{\beta}$ is topologically free,  then $C^*(B,\beta)$ is simple.
 \end{cor}
\begin{rem}
 If $\beta$ is injective topological freeness of $\widehat{\beta}$ implies that no power $\alpha^n$, $n>0$, is inner. Hence  Corollary \ref{cor for simplicity} is consistent with  Theorem \ref{theorem on simplicity} ii),  see also  Remark \ref{remark on pederesen olesen}.
\end{rem}
We finish this subsection with an application to general $C^*$-dynamical systems.
 \begin{lem}\label{lemma Stanislawa Lema}
 Let $(A,\alpha)$ be a $C^*$-dynamical system, $J$  an ideal in $(\ker\alpha)^\bot$ and $(B,\beta)$ the natural reversible  $J$-extension of $(A,\alpha)$. Then 
 $\alpha$ is pointwise quasinilpotent if and only if  $\beta$ is pointwise quasinilpotent.
 \end{lem}
 \begin{proof}
 With the notation from the proof of Theorem \ref{granica prosta zepsute twierdzenie},  $\beta$ is determined by the formulas $\beta(\phi_0(a_0))=\phi_0(\alpha(a_0))$ and  
 $$
 \beta(\phi_n(q(a_{0})\oplus ... \oplus q(a_{n-1}) \oplus a_{n}))=\phi_{n-1}(q(a_{1}))\oplus ... \oplus q(a_{n-1}) \oplus a_{n}),
 $$
 for $n>0$,  $a_k\in A_k$, $k=0,...,n$. Thus  $\beta$  is pointwise quasinilpotent on $\bigcup_{n\in \N} \phi_n(B_n)$ if and only if $\alpha$ is pointwise quasinilpotent on $A$. The assertion follows by the density of $\bigcup_{n\in \N} \phi_n(B_n)$ in $B$ .
\end{proof}

\begin{cor}\label{cor for freeness of quasinilpotents} 
 Let $(A,\alpha)$ be a $C^*$-dynamical system and $J$ an ideal in $(\ker\alpha)^\bot$.   If  $\alpha$ is pointwise quasinilpotent then all  ideals in $C^*(A,\alpha,J)$ are gauge-invariant. 
 \end{cor}
\begin{proof}
In view of Corollary \ref{freeness implies gauge-invariance} and Lemma \ref{lemma Stanislawa Lema} it suffices to note that if  a reversible $C^*$-dynamical system $(B,\beta)$ is such that  $\beta$ is pointwise quasinilpotent, then $\widehat{\beta}$ is  free. But if $\widehat{\beta}$ is not free then $\widehat{\beta_{\I}}$, for a certain subsystem $(B/\I, \beta_{\I})$ of $(B,\beta)$, is not topologically free. This implies that there exists $b \in B/\I$, a representation $\pi\in\widehat{B/\I}$ and a number $n$ such that $0\neq \|\pi(b)\|= \|\pi(\beta_{\I}^{nk}(b))\|$ for all $k\in \N$. Hence $\beta$ can not be pointwise quasinilpotent.
\end{proof} 
\subsection{Commutative $C^*$-dynamical systems}\label{Commutative $C^*$-dynamical systems} We fix a  locally Hausdorff space $X$ and a $C^*$-dynamical system $(A,\alpha)$ where $A=C_0(X)$ is the algebra of continuous functions on $X$ that vanish at infinity. We also fix an ideal $J\subseteq (\ker\alpha)^\bot$ and denote by $(B,\beta)$ the natural reversible $J$-extension of $(A,\alpha)$.
Clearly, $B$ is also commutative. We will  identify $B$  with $C_0(\X)$ for some locally compact Hausdorff space $\X$. 
Then, for $a\in A$ and $b\in B$, we have
$$
\alpha(a)(x)=
\begin{cases}
a(\varphi(x)), &x\in \Delta ,
\\
0, &x\notin \Delta , 
\end{cases}, \qquad  \beta(b)(x)=
\begin{cases}
b(\tphi(\x)), &\x\in \TDelta ,
\\
0, &\x\notin \TDelta, 
\end{cases}
$$
where  $\varphi:\Delta\to X$ and $\tphi: \TDelta\to \X$ are proper continuous mappings defined respectively on clopen subsets $\Delta\subseteq X$ and $\TDelta\subseteq \X$ (they  are closed because $\alpha$ and $\beta$ are extendible). One deduces from  Propositions \ref{Proposition on transfers}, \ref{lemma for extensions of transfers}  that   $\tphi$ is a homeomorphism onto a clopen set $\tphi(\TDelta)\subseteq \X$ and $\beta_*$ is actually an extendible endomorphism given by the similar formula as $\beta$ but with  $\tphi$ replaced with $\tphi^{-1}$.  

We dualize our construction of $(B,\beta)$ to obtain a description of $(\X,\tphi)$ in  terms of $(X,\varphi)$. This end let $Y \subseteq X$ be  the hull of the ideal $J$, that is  $Y$ is the closed set  such that 
$$%\begin{equation}
J=C_0(X\setminus Y).
$$%\end{equation}
Note that the relation $J\subseteq (\ker\alpha)^\bot$ is equivalent to the equality  $ Y\cup \varphi(\Delta)=X$.
In the unital case the following description, using a different approach, was obtained in  \cite[Theorem 3.5]{kwa-logist} and under an additional assumption in \cite[Theorem 3.5]{maxid}.

\begin{prop}\label{reversible extension on topological level} Up to a homeomorphism   $\X$ is the following  subspace  of the product space $\prod_{n\in \N} (X\cup \{0\})$, where $\{0\}$ is an abstract clopen singleton:
$$
\X=\bigcup_{N=0}^{\infty} X_N\cup X_\infty
$$
where
$$
X_N=\{(x_0,x_1,...,x_N,0,...): x_n\in \Delta,\, \varphi(x_{n})=x_{n-1},\,\, n=1,...,N,\,x_N\in
Y\},
$$
and 
$$
X_\infty=\{(x_0,x_1,...): x_n\in \Delta,\,
\varphi(x_{n})=x_{n-1},\, \,n\geqslant 1\}.
$$
Then $
\TDelta=\{(x_0,x_1,...)\in \X: x_0 \in \Delta\}$, $\tphi(\TDelta)=\{(x_0,x_1,...)\in \X: x_1 \neq 0\}
$,
and  
$$
\tphi(x_0,x_1,...)=(\varphi(x_0),x_0,x_1,...),\qquad \tphi^{-1}(x_0,x_1,...)=(x_1,...).
$$
\end{prop}
\begin{proof}
Let $\Delta_n:=\varphi^{-n}(\Delta)$ be the  natural domain of the partial mapping $\varphi^n$. Dualizing  diagram \eqref{direct limit diagram to be dualized} we see that  direct sequence \eqref{ciag prosty C*-algebr}  dualizes to the following inverse sequences
$$
\begin{xy}
\xymatrix@C=2pt@R=24pt{
 \widehat{B}_{0}   & \qquad =\qquad &     X
   \\
 \widehat{B}_{1} \ar[u]_{\widehat{T}_1} & \qquad =\qquad &    Y \ar[u]_{id}& \sqcup & \Delta_1  \ar[ull]_{\varphi}
      \\
  \widehat{B}_{2}  \ar[u]_{\widehat{T}_2} & \qquad =\qquad  &      Y \ar[u]_{id}& \sqcup &  Y\cap \Delta_1 \ar[u]_{id} & \sqcup & \Delta_2 \ar[ull]_{\varphi}
      \\
   \ar@{.}[u]_{\widehat{T}_3}  &  &     \ar@{.}[u]_{id}& &  \ar@{.}[u]_{id}& & \ar@{.}[u]_{id}&  & \qquad  \ar@{.}[ull]_{\varphi}
        }
  \end{xy}
$$
where  $\sqcup$ denotes the direct sum of topological spaces.
Now it is easy to  see that the inverse limit $\underleftarrow{\,\,\lim\,}\{\widehat{B_n},\widehat{T}_n\}$ can be identified with the space $\X$ described in the assertion. For instance,  if  $\x=(x_0,x_1,x_2,...)$ is an element of the inverse limit, then we have two alternatives: either every coordinate of $\x$ `lies on the diagonal' of the above diagram, and then   $\x$  identifies as the element of $X_\infty$;  
or  only $N$ first coordinates of $\x$ `lie on the diagonal' and next ones
`slide off the diagonal' fixing at $x_N$ - then we identify  $\x$  with the element of $X_N$ by replacing  the tail of  $x_N$'s  with tail of zeros.
We leave the remaining tedious details to the reader.\end{proof}

We extend Definition \ref{topolo disco} to not (necessarily) injective maps.
  \begin{defn}\label{topological freeness for maps}
We say that a periodic orbit $\OO=\{x, \varphi(x),..., \varphi^{n}(x)\}$ of a periodic point $x$    \emph{has an entrance} $y\in \Delta$  if $y\notin \OO$ and $\varphi(y)\in \OO$.  The partial mapping $\varphi$ is \emph{topologically free} if the set of all periodic points whose orbits have no entrances
 has empty interior. 
\end{defn} 
We will also need  a relative version of topological freeness which coincides with the unrelative one  when  $Y=\overline{X\setminus\varphi(\Delta)}$ is the smallest set under consideration (recall that we must have $Y\cup \varphi(\Delta)=X$).
\begin{defn}
The partial mapping $\varphi$ is said to be \emph{topologically free outside $Y$} if the set of periodic points  whose orbits do not intersect $Y$ and have no entrances have empty interior. 
\end{defn}
\begin{lem}\label{equivalence of topological freedom}
Let $(X,\varphi)$ and $(\X,\tphi)$ be as in Proposition \ref{reversible extension on topological level}. The partial homeomorphism $\tphi$ is topologically free  if and only if   $\varphi$ is topologically free outside $Y$.
\end{lem}  
\begin{proof}
Let $n>0$. By Baire theorem it suffices to show that the set  $\F_n:=\{\x: \tphi^n(\x)=\x\}$ has empty interior in $\X$ if and only if the set
$$
F_n:=\{x: \varphi^n(x)=x \textrm{ and }\varphi^{-1}(\varphi^k(x))=\{\varphi^{k-1}(x)\}\subseteq X\setminus Y,\textrm{ for all }k = 1, . . . , n\}
$$ 
has  empty interior in $X$. But if $U$ is a non-empty open subset of $F_n$ then our description of $(\X,\tphi)$ implies that 
$
\U:=\{x\in \X: x_0 \in U\} 
$ is non-empty open subset of $\F_n$. Conversely,  if $\F_n$ has a non-empty interior then $\F_n$ contains a non-empty set of the form  $\U:=\{x\in \X: x_n \in U\}$ for an open set $U$ in $X$. The inclusion $\U\subseteq \F_n$ forces $U$ to be contained in $F_n$.
\end{proof}

Combining the above lemma and Theorem \ref{topological freeness for reversible systems theorem}, modulo the last part of Proposition \ref{representations vs reversible extensions} and Example \ref{orbit representation}, we get 
 \begin{prop}
The following conditions are equivalent:
\begin{itemize}
\item[i)] the partial map $\varphi$ is topologically free outside  $Y$ (the hull of the ideal $J$),
\item[ii)] any faithful representation $(\pi,U)$ of $(A,\alpha)$ such that $
J=\{a\in A: U^*U\pi(a)=\pi(a)\}
$ 
give rises to a faithful representation $C^*(A,\alpha,J)$
\end{itemize}
%Note that if $J=(\ker\alpha)^\bot$ is complemented, then $u^*u\in \iota(A)$.
\end{prop}

\begin{rem} If $Y=X$ then  any $\varphi$ is (trivially) topologically free outside $Y$. Hence regardless of $(A,\alpha)$ any faithful  representation $(\pi,U)$ of $(A,\alpha)$ such that $\{a\in A: U^*U\pi(a)=\pi(a)\}=\{0\}$, integrates to the isomorphism from $C^*(A,\alpha,\{0\})$ onto $C^*(\pi,U)$. This is a resemblance of much more general facts whose prototype is Coburn's uniqueness theorem.
\end{rem}
We define objects dual to $J$-pairs and $J$-invariant ideals as follows
\begin{defn}\label{invariants definition} A set   $V\subseteq X$ is \emph{positively invariant} under $\varphi$ if  $\varphi(V\cap \Delta)\subseteq V$, and $V$ is \emph{$Y$-negatively invariant} if $V\subseteq Y\cup\varphi(V\cap \Delta)$. If $V$ is both positively and $Y$-negatively invariant, we  call  it   \emph{$Y$-invariant}.
 A pair $(V,V')$ of closed subsets of $X$ such that 
$$
V \textrm{ is positively } \varphi\textrm{-invariant,}\quad  V'\subseteq Y \quad \textrm{and}\quad  V'\cup \varphi(V\cap \Delta)=V
$$  
will be called a $Y$-pair for $(X,\varphi)$. We equip  the set $Y$-pairs  with natural partial order: $(V_1, V_1') \subseteq (V_2,V_2')$ $\stackrel{def}\Longleftrightarrow$  $V_1\subseteq V_2$ and $V_1'\subseteq V_2'$.
\end{defn}

\begin{prop}\label{proposition on ideals in commutative}
All ideals in $C^*(A,\alpha,J)$ are gauge-invariant if and only if $\varphi$ has no periodic points. In general, the relations 
$$
C_0(X\setminus V)=\{a \in A: \iota(a) \in \II\}, \qquad C_0(X\setminus V')=\{a\in A: (1-u^*u) \iota(a) \in \II\},
$$
establish an anti-isomorphism between the lattices of gauge-invariant ideals  $\II$ in $C^*(A,\alpha,J)$ and $Y$ pairs $(V,V')$ for $(X,\varphi)$; it restricts to the anti-isomorphism between the lattices of all ideals $\II$ in $C^*(A,\alpha,J)$ generated by their intersection with $\iota(A)$ and $Y$-invariant closed sets $V$ for $(X,\varphi)$.
 \end{prop}
\begin{proof}
Since $\X$ is Hausdorff, freeness of $\tphi$ (Definition \ref{phi invariance}) is equivalent to nonexistence of periodic points for $\tphi$. In view of our description of $\tphi$ the latter  is equivalent to nonexistence  of periodic points for $\varphi$. Thus the initial part of the assertion follows from Corollary \ref{freeness implies gauge-invariance}. The second part follows from Theorem \ref{lattices descriptions main thm} modulo a simple remark that a pair $(C_0(X\setminus V),C_0(X\setminus V'))$ forms  a $J$-pair for $(A,\alpha)$ if and only if $(V,V')$ form as a $Y$-pair for $(X,\varphi)$.
\end{proof}
We  say  that a (full) mapping $\varphi:X\to X$ is \emph{minimal} provided there are no nontrivial closed subsets $V$ of $X$ such that $\varphi(V)=V$. If $X$ is compact then $\varphi$  is minimal if and only if  there are no nontrivial closed positively invariant subsets $V$ of $X$, i.e. such that $\varphi(V)\subseteq V$. However,  when $X$ is not compact the  latter condition is much stronger then minimality. In fact, by \cite[Theorem B]{Gottschalk}, nonexistence of nontrivial closed positively invariant subsets in $(X,\varphi)$ forces $X$ to be compact. On the other hand,  there are interesting minimal mappings on non-compact spaces, cf. for instance \cite{Danilenko}.
\begin{thm}\label{simplicity in commutative case}
Let $(A,\alpha)$ be a commutative $C^*$-dynamical system and $(X,\varphi)$ its dual partial dynamical system. The  crossed product $C^*(A,\alpha)$ is simple if and only one of the two possible cases hold:
\begin{itemize}
\item[i)] $X$ is discrete, $\varphi$ is injective and $X$ compose of one non-periodic orbit $\OO$. 
In this case $C^*(A,\alpha)\cong \K(\ell^2(\OO))$ is the algebra of  compact operators on the $|\OO|$-dimensional Hilbert space.
\item[ii)] $X$ is not discrete   and $\varphi:X\to X$ is a minimal surjection. In this case $C^*(A,\alpha)\cong C^*(\X,\tphi)$ is the $C^*$-algebra of the  minimal homeomorphism $\tphi$ induced by $\varphi$ on the inverse limit space $\X$.
\end{itemize}
\end{thm}
\begin{proof} Let  $C^*(A,\alpha)$ be simple. Assume  that $\varphi:\Delta \to \varphi(\Delta)$ is a homeomorphism and $\varphi(\Delta)\subseteq X$ is  clopen  (think of it as of  the system $(\X,\tphi)$ described in Proposition \ref{reversible extension on topological level}). 
Evidently,  any orbit $\OO$ and hence its closure is a $\varphi$-invariant set, cf. Definition \ref{phi invariance}. Therefore,    by Corollary \ref{freeness implies gauge-invariance}, $X$ is the closure of $\OO$ and $\OO$ is not a periodic orbit.

Suppose  that $X$ is discrete. Then   $X=\OO$. Since $\OO$ is not periodic,   $(X,\varphi)$ is up to conjugacy either a truncated shift on $\{1,...,n\}$, one sided shift on $\N$, or a two-sided shift on $\Z$. In each of these cases   representations described in  Example \ref{orbit representation} yield  the isomorphism $C^*(A,\alpha)\cong \K(\ell^2(\OO))$, cf.  Theorem \ref{topological freeness for reversible systems theorem}.

Now let $X$ be arbitrary. We claim that   if  $\Delta\neq X$ then $X$ is  discrete. Indeed, if $X\setminus \Delta=\overline{\OO}\setminus \Delta\neq \emptyset$ then $\OO\setminus \Delta =\{x_0\}$ must be a singletone. As  $\OO\subseteq  \Delta \cup \{x_0\}$ we actually  have $X=\overline{\OO}= \Delta \cup \{x_0\}$. Therefore  $\{x_0\}$ is clopen in $X$. Consequently,   $\OO= \bigcup_{n\in \N} \varphi^{-n}(x_0)$ is an open discrete set. Clearly it is $\varphi$-invariant, and thus by minimality $X=\OO=\bigcup_{n\in \N} \varphi^{-n}(x_0)$ is discrete. 

Applying the above argument to the system $(X,\varphi^{-1})$ we see that $X$ is  discrete also when  $\varphi(\Delta)\neq X$. Thus if $X$ is not discrete then  $\Delta=\varphi(\Delta)=X$ and this finishes the proof in the reversible case. 

Now suppose $(X,\varphi)$ is a general partial dynamical system (with $\varphi$ not necessarily injective). Then as we have shown the  reversible system $(\X,\tphi)$ described in Proposition \ref{reversible extension on topological level} satisfies the assertion. A moment of thought leads to the conclusion that if $(\X,\tphi)$ is as described in the case i) then  $(X,\varphi)=(\X,\tphi)$. Similarly,  $\tphi:\X\to \X$ is a full minimal homeomorphism iff  $\varphi:X\to X$ is a full minimal surjection.
\end{proof}

\section{Appendix: $C^*(A,\alpha,J)$ viewed as a relative Cuntz-Pimsner algebra}\label{appendix}

  In this section we briefly outline how to obtain our description of the lattice of gauge-invariant ideals    using general machinery of relative Cuntz-Pimsner algebras. For more details, concerning  the latter  we refer the reader to  \cite{ms}, \cite{kwa-leb}, \cite{katsura1}.

A $C^*$-correspondence over a $C^*$-algebra $A$ is a right Hilbert $A$-module $X$ with a left action $\phi_X:A\to \L(X)$ of $A$ on $X$ via adjointable operators. We let $J(X):=\phi^{-1}(\K(X))$  to be the ideal in $A$ consisting of  elements that act from the left on $X$ as generalized compact operators. 
For any ideal $J$ in $J(X)$ the  relative Cuntz-Pimsner algebra $\OO(J,X)$ is constructed as a quotient of  the $C^*$-algebra generated by  Fock representation of $X$, see \cite[Definition 2.18]{ms} or \cite[Definition 4.9]{kwa-leb}. The $C^*$-algebra $\OO(J,X)$ is universal with respect to appropriately defined representations  of $X$, see \cite[Remark 1.4]{fmr} or \cite[Proposition 4.10]{kwa-leb}. It is equipped with a gauge circle action which acts as identity on the image of $A$ in $\OO(J,X)$. Katsura, in \cite{ka3}, described ideals in $\OO(J,X)$ that are invariant under this action in the following way.

For any ideal $I$ in $A$ we define two another ideals 
$$
X(I):=\clsp\{\langle x,a\cdot y\rangle_A\in A: a\in I,\,\, x,y\in X  \}, 
$$ 
$$
X^{-1}(I):=\{a\in A: \langle x,a\cdot y\rangle_A \in I \textrm{ for all } x,y\in X  \}. 
$$
If $X(I)\subseteq I$, then the ideal  $I$ is said to be \emph{positively invariant}, \cite[Definition 4.8]{ka3}. 
For any positively invariant ideal $I$ we have a naturally defined quotient $C^*$-correspondence $X_I=X/XI$ over $A/I$. Denoting by $q_I:A\to A/I$ the quotient map one puts
$$
J_X(I):=\{a\in A: \phi_{X_I}(q_I(a))\in \K(X_I), \, \, aX^{-1}(I) \subseteq I\}.
$$   
\begin{defn}[Definition 5.6 in \cite{ka3}]Let $X$ be a $C^*$-correspondence over a $C^*$-algebra $A$. A \emph{$T$-pair} of
$X$ is a pair $(I,I')$ of ideals $I$, $I'$ of $A$ such that $I$ is positively invariant and
$I\subseteq I'\subseteq  J_X(I)$.
\end{defn}
The content of \cite[Proposition 11.9]{ka3}  is the following:
\begin{thm}\label{thm for referee}
Let $X$ be a $C^*$-correspondence over a $C^*$-algebra $A$, and $J$ be
an ideal of $A$ contained in $J(X)$. Then there exists a one-to-one correspondence
between the set of all gauge-invariant ideals of $O(J,X)$ and the set of all $T$-pairs $(I, I')$ of $X$ satisfying $J\subseteq I'$, which preserves inclusions and intersections.
\end{thm}
 Let us now fix a $C^*$-dynamical system $(A,\alpha)$ and an ideal $J\subseteq (\ker\alpha)^\bot$.  We associate to $(A,\alpha)$  a $C^*$-correspondence  $X_\alpha$ which as a vector space is equal to $\overline{\alpha}(1)A$  and the other operations are given by   
$$
a\cdot x := \alpha(a)x \qquad x\cdot a := xa,\qquad  \langle x,y\rangle_A:= x^*y, \qquad a\in A,\,\, x,y\in X_\alpha.
$$
One can show that $J(X_\alpha)=A$.  Hence we can consider the relative Cuntz-Pimsner algebra $\OO(J,X_\alpha)$. Moreover, there is a natural isomorphism 
\begin{equation}\label{isomorphism for referee}
C^*(A,\alpha,J)\cong \OO(J,X_\alpha)
\end{equation}
which maps $\iota(A)$ and $u^*\iota(A)$, respectively,  onto the image of $A$ and $X_\alpha$ in  $\OO(J,X_\alpha)$. The latter fact was noticed in \cite[Example 1.6]{fmr} and \cite[Corollary 4.14]{kwa-leb} under the assumption that $A$ is unital, but the arguments carry out easily to our  more general situation. In particular,  \eqref{isomorphism for referee} induces a lattice isomorphism between the gauge-invariant ideals in $C^*(A,\alpha,J)$ and  $\OO(J,X_\alpha)$.

Clearly, for any ideal $I$ in $A$ we have
$
X_\alpha(I)=A\alpha(I)A$ and  $X_\alpha^{-1}(I)=\alpha^{-1}(I).
$
In particular, $I$ is positively invariant for $X$ if and only if $I$ is positively invariant in $(A,\alpha)$.
If $I$ is positively invariant then the quotient $C^*$-correspondence $X/XI$ can be identified with the $C^*$-correspondence $X_{\alpha_I}$ associated to the quotient system $(A/I,\alpha_I)$. Hence $J(X_I)=A/I$ and we get 
$
J_{X_\alpha}(I)=\{a\in A: a\alpha^{-1}(I)\subseteq I\}.
$
Accordingly, if $I$ and $I'$ are ideals in $A$ we have 
$$
(I,I') \text{ is a }T\text{-pair with } J\subseteq I' \,\, \Longleftrightarrow\,\,  (I,I') \text{ is a }J\text{-pair for } (A,\alpha).
$$
In this way we can infer from Theorem \ref{thm for referee} the  first part of Theorem \ref{lattices descriptions main thm}:
\begin{cor}
Let $(A,\alpha)$ be a $C^*$-dynamical system and $J$ an ideal in $(\ker\alpha)^\bot$. We have an order isomorphism  from  the lattice of gauge-invariant  ideals $\II$ in $C^*( A,\alpha,J)$ onto the lattice of  $J$-pairs $(I,I')$ for $(A,\al)$.
\end{cor}
 In a similar way, using isomorphism  \eqref{isomorphism for referee}, one can apply to $C^*(A,\alpha,J)$ other general results for relative Cuntz-Pimsner algebras. For instance,  it follows from \cite[Theorem 7.1]{katsura} and \cite[Theorem 7.2]{katsura}, respectively that
$$
A \textrm{ is exact } \Longleftrightarrow\,\,\, C^*(A,\alpha,J)\textrm{ is exact}.
$$
$$
A \textrm{ is nuclear } \Longrightarrow\,\,\, C^*(A,\alpha,J)\textrm{ is nuclear}.
$$
If $A$ is separable, then $C^*(A,\alpha,J)$ is separable and by the argument leading to  \cite[Proposition 8.8]{katsura}  we have  
$$
\textrm{both } A \textrm{ and  } J \textrm{ satisfy the UCT }\,\,\, \Longrightarrow\,\,\, C^*(A,\alpha,J)\textrm{ satisfy the UCT}.
$$

\end{document}